\documentclass[a4paper,english,leqno]{article}
\usepackage[T1]{fontenc}
\usepackage[latin9]{inputenc}
\usepackage{amsthm} 
\usepackage{amsmath}
\usepackage{amssymb,amsfonts,amscd,mathrsfs}
\usepackage{esint}
\makeatletter

\pdfpageheight\paperheight
\pdfpagewidth\paperwidth

\theoremstyle{plain}
\newtheorem{thm}{\protect\theoremname}[section]
  \theoremstyle{plain}
  \newtheorem{cor}[thm]{\protect\corollaryname}
  \theoremstyle{plain}
  \newtheorem{definition}[thm]{\protect\definitionname}
  \theoremstyle{plain}
  \newtheorem{lem}[thm]{\protect\lemmaname}
  \theoremstyle{plain}
  \newtheorem{rem}[thm]{\protect\remarkname}
  \theoremstyle{plain}
  \newtheorem{prop}[thm]{\protect\propositionname}

\usepackage{amssymb,amsfonts,amsmath,amsthm,amscd}
\usepackage[T1]{fontenc}
\usepackage[latin9]{inputenc}
\usepackage{geometry}
\usepackage{xcolor}
\usepackage{pdfcolmk}
\usepackage{refstyle}
\usepackage{textcomp}
\usepackage{esint}
\PassOptionsToPackage{normalem}{ulem}
\usepackage{ulem}

\def\IR{{\mathbb R}}

\makeatletter

\AtBeginDocument{}
\RS@ifundefined{subref}
  {\def\RSsubtxt{section~}\newref{sub}{name = \RSsubtxt}}
  {}
\RS@ifundefined{thmref}
  {\def\RSthmtxt{theorem~}\newref{thm}{name = \RSthmtxt}}
  {}
\RS@ifundefined{lemref}
  {\def\RSlemtxt{lemma~}\newref{lem}{name = \RSlemtxt}}
  {}

\providecolor{lyxadded}{rgb}{0,0,1}
\providecolor{lyxdeleted}{rgb}{1,0,0}

\numberwithin{equation}{section}
\numberwithin{figure}{section}
\theoremstyle{plain}

\newtheorem{remark}[thm]{Remark}

\addtocounter{section}{-1}
\date{}
\author{Xinyi Li}
\usepackage{babel}
  \providecommand{\corollaryname}{Corollary}
  \providecommand{\remarkname}{Remark}
    \providecommand{\definitionname}{Definition}
  \providecommand{\lemmaname}{Lemma}
  \providecommand{\propositionname}{Proposition}
\providecommand{\theoremname}{Theorem}

\begin{document}

%

\title{\textbf{\normalsize{PERCOLATIVE PROPERTIES OF BROWNIAN INTERLACEMENTS\\ AND ITS VACANT SET}}}

\maketitle
\vspace{-0.7cm}


\vspace{1cm}

\begin{abstract}
 In this article we investigate the percolative properties of Brownian interlacements, a model introduced by Alain-Sol Sznitman in \cite{SznBI}, and show that: the interlacement set is ``well-connected'', i.e., any two ``sausages'' in $d$-dimensional Brownian interlacements, $d\geq 3$, can be connected via no more than $\lceil (d-4)/2 \rceil$ intermediate sausages almost surely; while the vacant set undergoes a non-trivial percolation phase transition when the level parameter varies. 
\end{abstract}
\vspace{11.0cm}

\date{\begin{flushright}{\small September 2019}\end{flushright}}

\vfill \noindent {\small -------------------------------- \\ Beijing International Center for Mathematical Research, Peking University, Beijing 100871, China} \thispagestyle{empty} 

\pagebreak{}\newpage{}

\pagebreak{}

\newpage{}

\mbox{} \thispagestyle{empty} \newpage

\section{Introduction}\label{sec:intro}

In this article, we investigate various aspects of the percolative properties of Brownian interlacements, and show that the interlacements are well-connected and that the vacant set undergoes a non-trivial phase transition. 

The model of Brownian interlacements, recently introduced by Sznitman in \cite{SznBI}, is the continuous counterpart of random interlacements, a model that has already attracted a lot of attention and has been relatively thoroughly studied (see \cite{SznAnnal} for the seminal paper on this model and see \cite{Mousquetaires} and \cite{Cernynotes} for a comprehensive introduction). Roughly speaking, Brownian interlacements can be described as a certain Poissonian cloud of doubly-infinite continuous trajectories in the $d$-dimensional Euclidean space, $d\geq3$, with the intensity measure governed by a parameter $\alpha>0$. We are interested in both the {\it interlacement set}, which is an $r$-enlargement (sometimes colloquially referred to as ``the sausages'') of the union of the trace in the aforementioned cloud (of trajectories), for some $r>0$, and the {\it vacant set}, which is the complement of the interlacement set. 

Brownian interlacements bear similar properties, for instance long-range dependence, to random interlacements, due to similarities in the construction. Moreover, this model plays a crucial role in both the study of the limiting behaviors of various aspects of random interlacements (see for example \cite{SznBI} and \cite{LiSzni13}),  and the interconnection of random interlacements, loop soups, and Gaussian free fields. Brownian interlacements, as a model of continuous percolation, could also shed light on the study of other models. For example, the visibility in the vacant set of Brownian interlacements is studied and compared with that of the Brownian excursion process in \cite{Visibility}.

\medskip

We now describe the model and our results in a more precise fashion. Readers are referred to Section \ref{notations} for notations and definitions. We consider Brownian interlacements on $\mathbb{R}^d$, $d\geq 3$. We denote by $\mathbb{P}$ 
the canonical law of Brownian interlacements 
and by $\mathcal{I}^\alpha_r$ (resp.~$\mathcal{V}^\alpha_r$) the corresponding interlacement set (resp.~vacant set) 
 at level $\alpha\geq 0$ with radius $r\geq 0$, which is $\mathbb{P}$-a.s.~closed (resp. open).
\medskip

Let us look at the interlacement set first. 

As pointed out in (2.36) in \cite{SznBI}, it is presently known that for any $\alpha>0$ the Brownian fabric (i.e.~$\mathcal{I}^{\alpha}_0$) is connected when $d=3$ and disconnected when $d\geq 4$.
However, despite its name ``Brownian interlacements'', not much is a priori known about how the trajectories are actually interlaced: for example, the connectedness of $\mathcal{I}^\alpha_r$ for $\alpha,r>0$ in dimension $4$ and higher does not follow trivially from the definition (however it is a simple corollary of Theorem \ref{graphdistance}). 

In this work, we show that, for $d\geq 3$, and for all $\alpha,r>0$, 
Brownian interlacements are well-connected in the following sense: two sausages in the interlacements can be connected via no more than $s_d-1$ intermediate sausages, where
\begin{equation}\label{eq:sddef}
s_d\stackrel{\triangle}{=}\left\lceil \frac{d-2}{2} \right\rceil
\end{equation}
 ($\lceil \frac{d-2}{2} \rceil$ stands for the smallest integer greater or equal to $\frac{d-2}{2}$).
 
We now phrase this result in a more precise fashion. Let $W$ stand for the collection of doubly-infinite continuous trajectories modulo time shift of  Brownian interlacements 
at level $\alpha$. 
We associate with $W$ and $r>0$ a graph $G_{\alpha,r}= (V,E)$ where each  vertex in $V$ corresponds to a trajectory in $W$, 
and the set of edges $E$ consists of pairs of vertices whose corresponding sausages of radius $r$ intersect with each other.
Let ${\rm diam}(G_{\alpha,r})$ stand  for the diameter of $G_{\alpha,r}$. The result regarding the connectivity of interlacements is summarized in the following theorem.

\begin{thm}\label{graphdistance}
Let $G_{\alpha,r}$ be the graph defined as above. For all $\alpha,r>0$,
\begin{equation}
\mathbb{P}[{\rm diam}(G_{\alpha,r})=s_d]=1.
\end{equation}
\end{thm}
As a corollary, we  obtain the connectedness of Brownian interlacements.
\begin{cor}
For all $\alpha,r>0$, the interlacement set $\mathcal{I}^\alpha_r$ is $\mathbb{P}$-a.s.~connected.
\end{cor}

We now make a few comments on Theorem \ref{graphdistance}. The formulation of this problem and the strategy of proof are inspired by \cite{GDRI}, which treated graph distance problem on random interlacements and obtained the same graph diameter as (\ref{eq:sddef}). It is worth mentioning that some of the methods and techniques used in \cite{GDRI} can also be adapted to serve as the backbone in the solution to other problems, such as \cite{RS2011}, \cite{RS2013} and \cite{CePo}. It is also worth mentioning that in the case of random interlacements, the same result can be proved through an essentially different approach, see \cite{PrTy}, which involves the notion of  ``stochastic dimensions''. However this notion is only defined on the discrete lattice and there is no adequate continuous equivalent.

\medskip

We  then turn to the vacant set. In this work, we show that for any $r>0$, the vacant set $\mathcal{V}^\alpha_r$ undergoes a non-trivial percolation phase transition. More precisely, we have 
 the following theorem (notice that  the vacant set is ``monotonously decreasing'' with respect to $\alpha$, i.e., it is possible to construct Brownian interlacements simultaneously for levels $\alpha_1>\alpha_2$ and $r>0$ in such a way that $\mathcal{V}^{\alpha_1}_r\subset \mathcal{V}^{\alpha_2}_r$, see (\ref{eq:monotonecoupling})).

\begin{thm}\label{thm:phase transition}
There exists $0<\alpha_{1}^{*}(d)<\infty$, such that
\begin{equation}\label{eq:pt}
\begin{split}
&\mbox{$\mathcal{V}_{r}^{\alpha}$ percolates $\mathbb{P}$-a.s., when $\alpha<\alpha_{r}^{*}\stackrel{\triangle}=\alpha_{1}^{*}r^{2-d}$; and}\\
&\mbox{$\mathcal{V}_{r}^{\alpha}$ does not percolate $\mathbb{P}$-a.s., when $\alpha>\alpha_{r}^{*}$}.
\end{split}
\end{equation}
\end{thm}

We refer to the case $0<\alpha<\alpha^*_r$ the {\it supercritical regime} and the case $\alpha>\alpha^*_r$ the {\it subcritical regime}, which is in line with random interlacements.

We now make a few comments about this theorem. 

The precise relation between $\alpha^*_r$'s for different values of $r$, given in (\ref{eq:pt}), is due to the scaling property of Brownian interlacements (see (\ref{scaling})), which also implies that it suffices to study the phase transition with regard to one parameter only. 
 
The critical percolation threshold $\alpha^*_1$ could be related to some of the questions concerning the complement of the Wiener sausage wrapping on a unit $d$-dimensional torus discussed in \cite{GoDeH}, relevant in the local scale  $t^{-1/(d-2)}$ (i.e.~the local scale $\phi_{local}(t)$ in the terminology of (1.5) in \cite{GoDeH}, see also Section 1.6.3, ibid., especially (1.43)) when the Brownian motion on the unit torus runs over time $t$. It is plausible, yet not known at the moment, that $\alpha^*_1$ enters into play in the following way: when one runs Brownian motion on the unit $d$-dimensional torus for time $\alpha t$, and looks at the complement of the $t^{-1/(d-2)}$-neighborhood of the trajectory, then

- when $\alpha >\alpha^*_1$, for large $t$, there are only ``small'' components,
but 

- when $\alpha<\alpha^*_1$, for large $t$, there is a ``giant component'',

\noindent in analogy with what happens in the case of the discrete
$d$-dimensional torus of large side-length $n$, see \cite{RWfrag} and \cite{TeiWin}. Taking $t/\alpha$ to play the role of $t$, the same applies to the complement of the sausage of radius ${(\alpha/t)}^{1/(d-2)}$ 
of the Brownian motion in time $t$, on the unit torus,
for large $t$ (depending on $\alpha >\alpha^*_1$ or  $\alpha<\alpha^*_1$).

In the course of proving Theorem \ref{thm:phase transition}, we are also able to show that $\mathcal{V}^{\alpha}_r$ undergoes another phase transition with respect to connectivity. More precisely,  let
\begin{equation}\label{eq:astarstar}
\alpha_r^{**}=\inf\{\alpha\geq 0: 
\liminf_{L\to\infty} 
\mathbb{P} [\exists
\textrm{ continuous path in }{\cal V}_{r}^{\alpha}
\textrm{ connecting }B_\infty(0,L)
\textrm{ and }\partial B_\infty(0,2L)]=0\}
\end{equation}
stand for the critical level of sharp connectivity decay for $\mathcal{V}^\alpha_r$, where $B_{\infty}(0,L)$ stands for a ball centered at the origin of size $L$ under $l^\infty$-norm, then 
\begin{equation}\label{eq:starleqdoublestar}
\alpha^*_r\leq \alpha^{**}_r<\infty.
\end{equation}
It is hence a very natural question whether $\alpha^{*}_r$ actually coincides with $\alpha^{**}_r$, which would imply that the phase transition is sharp.
Notice the similarity between $\alpha^{**}_r$ and the critical parameter $u_{**}$ for random interlacements, whose definition first appeared in \cite{Szni09} and was later improved subsequently in \cite{Szni12} and \cite{TeiPop}. 
As the corresponding conjecture for random interlacements has been open for a long time, we do not expect  a quick answer to this question here. See Remark \ref{rem:endremark} for more discussions. 

We also refer to Remark \ref{rem:endremark} for further discussions, such as the open question of the uniqueness of percolation cluster in the supercritical regime and the existence of a critical threshold for percolation on a plane.  

\medskip

We now give some comments on the proofs. 

For Theorem \ref{graphdistance}, the strategy we employ is inspired by \cite{GDRI} and involves developing parallel estimates for Brownian interlacements. The lower bound of the graph diameter is essentially a convolution estimate of Green's function. The more involved upper bound is more technical than the discrete version, but generally follows from the same idea with similar capacity estimates. Pick a certain trajectory from the interlacement process, look at the corresponding vertex in the connectivity graph, and consider the union of sausages of trajectories that correspond to  vertices of distance at most $s$ from this vertex. We use induction to show that this union is ``$(2s+2)$-dimensional'' in terms of capacity and when $s=s_d-1$ it satuates the space in the sense that it will almost surely be hit by another trajectory from the interlacement process. This in turn gives the upper bound on the graph diameter. The above lines are of course mainly heuristic, to make sense of the above heuristics, a multi-scale analysis is employed. For more details, see the beginning of Section \ref{ss:ubgd}.

We now turn to Theorem \ref{thm:phase transition}. 
In this work we have chosen the combinatorial approach of \cite{Shortproof} instead of the standard route map in proving non-trivial phase transitions for interlacements, namely via the  ``sprinkling'' technique and decoupling inequalities (see \cite{Szni12} or Chapter 8 of \cite{Mousquetaires}), for the latter is lengthier and more involving 
(but yields more quantitative controls, for instance in the region
corresponding in our set-up to  $\alpha >  \alpha^{**}_r$). A more detailed explanation on the proof strategy can be found at the beginning of Section \ref{se:3}.

\medskip

We will now explain how this article is organized. In Section \ref{notations} we introduce notation and make a brief review on results concerning Brownian motion and its potential theory, the definition and basic properties of Brownian interlacements, renewal theory as well as other useful facts and tools. Section \ref{se:2} is devoted to the proof of Theorem \ref{graphdistance}. The lower bound on the graph distance is proved in Proposition \ref{prop:graphlb} and the upper bound on graph distance is proved in Proposition \ref{prop:gdub}.  In Section \ref{se:3} we prove Theorem \ref{thm:phase transition}. The dyadic trees are defined in Section \ref{se:3.1}, where some preliminary results are also stated. In Sections \ref{se:3.2} and \ref{se:3.3} we prove some preparatory results for the finiteness and positiveness of the percolation threshold  respectively,  and the proof of Theorem \ref{thm:phase transition} shall be completed in Section \ref{denouement}.

 Finally, we explain the convention in this work. We denote by $c$, $c'$, $c''$, $\overline{c},\ldots$ positive constants with values changing from place to place. Throughout the article, the constants depend on the dimension $d$. 
Unless otherwise stated, throughout the article we assume $d\ge3$.

\medskip

\noindent{\bf Acknowledgements.} The author wishes to express his gratitude to Alain-Sol Sznitman for suggesting these problems and for numerous valuable discussions and thank Art\"em Sapozhnikov and Ron Rosenthal for various useful discussions.

\section{Some useful facts}\label{notations}

In this section we introduce various notation and recall useful facts concerning Brownian
motion, its potential theory, Brownian interlacements and renewal theory.

\subsection{Basic notations}\label{se:basic}
In this subsection we introduce some useful notation. We write $\mathbb{N}=\{0,1,2,\ldots\}$ for the set of natural numbers
and write $\mathcal{B}(\mathbb{R}^d)$ for the collection of Borel sets in $\mathbb{R}^d$. We write $|\cdot|$ and $|\cdot|_{\infty}$
for the Euclidean and $l^{\infty}$-norms on $\mathbb{R}^{d}$. We
denote by $B(x,r)=\{y\in\mathbb{R}^{d};\:|x-y|\leq r\}$ (resp.~$B^\circ(x,r)=\{y\in\mathbb{R}^{d};\:|x-y|< r\}$) the closed (resp.~open) 
Euclidean ball of center $x$ and radius $r\geq0$, and when $A$ is a
subset of $\mathbb{R}^{d}$, we write $B(A,r)=\cup_{x\in A}B(x,r)$
for the union of all closed balls of radius $r$ and with center in $A$ and call it the $r$-sausage of $A$. We also write $B_{\infty}(x,r)=\{y\in\mathbb{R}^{d},\:|x-y|_{\infty}\leq r\}$ for
the closed $l^{\infty}$-ball of center $x$ and radius $r$.  
In particular, for the sake of convenience we write $B(R)=B(0,R)$ for short. 
When $U$ is a subset of $\mathbb{R}^{d}$,
we denote by $\partial U$ the boundary  of $U$ and we denote by $\mathrm{Volume}(U)$ the volume of $U$.

We will make repeated use of the following basic observation:
\begin{equation}\label{latticeobservation}
\mbox{for every }a\in\mathbb{R}^d,\mbox{ there exists }\tilde{a}\in\mathbb{Z}^d\mbox{ such that }|a-\tilde{a}|\leq\frac{\sqrt{d}}{2}.
\end{equation}

We call $\gamma:[0,1]\to\mathbb{R}^{d}$ (resp.~ $\widetilde{\gamma}:[0,\infty)\to\mathbb{R}^{d}$) a continuous path from $A\subset\mathbb{R}^d$ to $B\subset\mathbb{R}^d$ (resp.~infinity), if $\gamma$ is continuous, $\gamma(0)\in A$ and $\gamma(1)\in B$ (resp.~$\limsup_{t\to\infty}|\widetilde{\gamma}(t)|=\infty$), and also say that $\gamma$ connects $A$ and $B$ (resp.~infinity). With slight abuse of preciseness, when we mention a continuous path we sometimes actually mean its trace, i.e., $\gamma([0,1])$ or $\gamma([0,\infty))$ as a subset of $\mathbb{R}^{d}$. 
We say that $A\subset \mathbb{R}^d$ {\it percolates}, if $A$ contains an unbounded connected subset. If in addition $A$ is open, then $A$ percolates if and only if there exists a continuous $\gamma:[0,\infty)\to \mathbb{R}^d$ such that $\gamma[0,\infty)\subset A$ and $\limsup_{t\to\infty} |\gamma(t)|=\infty$.

We now turn to discrete paths. We call $\gamma:\{0,1,\ldots,n\} \to\mathbb{Z}^{d}$ a nearest
neighbor path (resp.~$*$-path) 
for all $k\in\{0,1,\ldots n\}$, $|\gamma(k+1)-\gamma(k)|=1$ (resp.~$|\gamma(k+1)-\gamma(k)|_\infty=1$).
By definition a nearest neighbor path is also a $*$-path. Again, we do not distinguish a discrete path from its trace as a subset of $\mathbb{Z}^d$.

\subsection{Brownian motion and its potential theory}\label{bmpt}

In this subsection we introduce our notation for Brownian motion and state some useful results on the potential theory of Brownian motion. 

We denote by $W$ the subspace of $C(\mathbb{R},\mathbb{R}^{d})$, which consists of continuous trajectories from $\mathbb{R}$ into
$\mathbb{R}^{d}$ tending to infinity at both plus and minus infinite
times. Similarly, we denote by $W_{+}$ the subspace of $C(\mathbb{R}^{+},\mathbb{R}^{d})$
of continuous trajectories from $\mathbb{R}^{+}$ to $\mathbb{R}^{d}$,
tending to infinity at infinite time. We write $X_{t}$, $t\in\mathbb{R}$ (resp.~$X_{t}$, $t\geq0$) for the canonical process, and denote by $\theta_{t}$, $t\in\mathbb{R}$ (resp.~$\theta_{t}$, $t\geq0$)
the canonical shifts. The spaces $W$ and $W_+$ are endowed with respective $\sigma$-algebras $\mathcal{W}$ and $\mathcal{W}_+$ generated by the canonical processes. For the convenience of notation, we sometimes write $X(t)$ instead of $X_t$. For an index set $I\subset \mathbb{R}$, we write 
\begin{equation}
X_I=\bigcup_{i\in I} \{X_i\}
\end{equation}
for the trace of $X_t$ on $I$.

When $F$ is a closed subset of $\mathbb{R}^{d}$ and $w$ is in $W_{+}$, we write $H_{F}(w)=\inf\{s\geq0,X_{s}(w)\in F\}$ and $\widetilde{H}_{F}(w)=\inf\{s>0,X_{s}(w)\in F\}$
for the respective entrance time and hitting time of $F$. When $U$ is an open subset of $\mathbb{R}^{d}$, we write $T_{U}(w)=\inf\{s\geq0,X_{s}(w)\notin U\}$
for the exit time of $U$.   When $w\in W$,
we define $H_{F}(w)$ and $T_{U}(w)$ similarly, replacing
the condition $s\geq0$ by $s\in\mathbb{R}$.

\medskip
Now we turn to Brownian motion and its potential theory.

Since $d\ge3$, and in this case Brownian motion on $\IR^{d}$ is transient, we
view $P_{y}$, the Wiener measure starting from $y\in\mathbb{R}^{d}$,
as defined on $(W_{+},{\cal W}_{+})$, and denote by $E_{y}$ for
the corresponding expectation. Moreover, if $\rho$ is a finite measure (not necessarily a probability measure)
on $\IR^{d}$, we denote by $P_{\rho}$ and $E_{\rho}$ the measure $\int_{x\in\mathbb{R}^d}\rho(dx)P_x$ (not necessarily a probability measure) and its corresponding ``expectation'' (i.e.~the integral with respect to the measure $P_\rho$).

\medskip{}
We write
\begin{equation}
p_{t}(x,x')=\frac{1}{(\sqrt{2\pi t})^{d}}\exp\big(-\frac{|x-x'|^{2}}{2t}\big)\mbox{ for $t>0$, $x,x'\in\mathbb{R}^d$,}\label{transition}
\end{equation} for the Brownian transition density. Accordingly we denote the Green function of Brownian motion by
\begin{equation}\label{Greendef}
g(y,y')=\int_{0}^{\infty}p_{t}(y,y')\, dt,\;\;\mbox{for}\; y,y'\in\mathbb{R}^{d}.
\end{equation}
It is a classical result that 
\begin{equation}\label{eq:greenvalue}
g(y,y')=c|y-y'|^{2-d}\quad\mbox{for }y,y'\in\mathbb{R}^d.
\end{equation}
For $t\geq 0$, we write $P^t$ for the Brownian semi-group operator on $L^1(\mathbb{R}^d)$. More precisely, for all $f\in L^1(\mathbb{R}^d)$, we define $P^{t}f:\mathbb{R}^{d}\to\mathbb{R}$ in the following manner:

\begin{equation}
P^{t}f(x)=\int_{\mathbb{R}^{d}}f(y)p_{t}(x,y)dy.
\end{equation}
We write $G$ for the respective Green operator:
\begin{equation}
Gf(x)=\int_{0}^{\infty}P^tf(x)dt.
\end{equation}

We now derive in Lemma \ref{lem:Ptfbound} an upper bound on the $L^\infty$-norm of $P^t f$ in terms of the $L^1$ and $L^\infty$-norms of $f$, and in Lemma \ref{lem:greenfunctiondetailedcontrol} an estimate on Green function and Wiener sausages which is somewhat similar in flavor to Lemma 5.5 in \cite{PoisatErhard} but  tailor-made for the proof of Proposition \ref{prop:CactusEmin}.

\begin{lem}\label{lem:Ptfbound}
For all $f\in L^{\infty}(\mathbb{R}^{d})\cap L^{1}(\mathbb{R}^{d})$,
one has 
\begin{equation}
||P^{t}f||_{\infty}\leq\frac{C\max(||f||_{1},||f||_{\infty})}{(t\vee1)^{d/2}}.\label{eq:Ptfbound}
\end{equation}
\end{lem}
\begin{proof}
When $t\geq1$, one has

\begin{equation}
|P^{t}f(x)| \stackrel{(\ref{transition})}{=}
\left|\int_{\mathbb{R}^{d}}\frac{1}{(\sqrt{2\pi t})^{d}}e^{-\frac{(x-y)^{2}}{2t}}f(y)dy\right|\leq\frac{1}{(\sqrt{2\pi t})^{d}}\int_{\mathbb{R}^{d}}|f(y)|dy=Ct^{-d/2}||f||_{1}.\label{eq:Ptfg1}
\end{equation}
When  $t<1$, we have
\begin{equation}
|P^{t}f(x)|\stackrel{(\ref{transition})}{=}\left|\int_{\mathbb{R}^{d}}\frac{1}{(\sqrt{2\pi t})^{d}}e^{-\frac{(x-y)^{2}}{2t}}f(y)dy\right|\leq||f||_{\infty}\int_{\mathbb{R}^{d}}\frac{1}{(\sqrt{2\pi t})^{d}}e^{-\frac{(x-y)^{2}}{2t}}dy=||f||_{\infty}.\label{eq:Ptfl1}
\end{equation}
The claim (\ref{eq:Ptfbound}) then follows from combining (\ref{eq:Ptfg1})
and (\ref{eq:Ptfl1}).
\end{proof}

\begin{lem}\label{lem:greenfunctiondetailedcontrol}
Let $d\geq5$ and $(z_{i})_{i\geq1}$ be a sequence of points in
$\mathbb{R}^{d}$. We consider $(X^{i}_t)_{t\geq 0}$, $i\geq1$, a sequence of independent Brownian motions on $\mathbb{R}^{d}$ with $X^{i}(0)=z_{i}$, $i\geq1$, and write $E$ for the expectation with respect to their joint law. For all $z\in\mathbb{R}^d$, let $f_z(\cdot)=1_{B(z,1)}(\cdot)$.
For $i,j=1,\ldots,M$, we write
\begin{equation}\label{eq:FLdef}
F_L(i,j)=\int_{L/2}^{L}\int_{L/2}^{L}\int_{\mathbb{R}^{d}\times\mathbb{R}^{d}}g(x,y)f_{X^{i}_{s}}(x)f_{X^{j}_{t}}(y)dxdydsdt.
\end{equation}
Then for all positive integers $M$ and for all
$L\geq2$, 
\begin{equation}\label{bigupperbound1}
E\left[\sum_{i,j=1}^{M}F_L(i,j)\right]\leq C(ML+M^{2}L^{3-d/2}).
\end{equation}
\end{lem}
\begin{proof}
We divide the summation into two cases, namely $i=j$ and $i\neq j$. To prove (\ref{bigupperbound1}), it suffices to prove that 
for all $i=1,\ldots,M$, 
\begin{equation}\label{stupid1}
E[F_L(i,i)]\leq cL
\end{equation}
and for all $i,j=1,\ldots,M$, $i\neq j$,
\begin{equation}\label{stupid2}
E[F_L(i,j)]\leq c'L^{3-d/2}.
\end{equation}
We first prove (\ref{stupid1}). For $f,g:\mathbb{R}^d\to\mathbb{R}$, let $\langle f,g\rangle$ stand for the inner product of $f$ and $g$. We then rewrite $E[F_L(i,i)]$ in
the form of semi-group operators:
\begin{eqnarray}
 E[F_L(i,i)]& = & 2E\left[\int_{L/2}^{L}\int_{s}^{L}\int_{\mathbb{R}^{d}\times\mathbb{R}^{d}}g(x,y)f_{X^{i}_{s}}(x)f_{X^{i}_t}(y)dxdydtds\right]\nonumber\\
&= & 2E\left[\int_{L/2}^{L}E^*\left[\int_{0}^{L-s}\int_{\mathbb{R}^{d}\times\mathbb{R}^{d}}g(x,y)f_{X^{i}_{s}}(x)f_{X^{i}_{s}+X^*_{t'}}(y)dydxdt'ds\right]\right]\\
 & = & 2E\left[E^* \left[ \int_{L/2}^{L}\int_{0}^{L-s}\langle f_{X^{i}_{s}},Gf_{X^{i}_{s}+X^*_{t'}}\rangle dt'ds\right]\right],\nonumber
\end{eqnarray}
where we denote by $X^*_t$ a Brownian motion started from $0\in\mathbb{R}^d$ which is independent from  $(X^{i}_t)_{t\geq 0}$, $i\geq1$, and write $E^*$ for its respective expectation.
We are now ready to show (\ref{stupid1}) with the help of (\ref{eq:Ptfbound}) from Lemma \ref{lem:Ptfbound}. Notice that by the observation $f_{a}(b)=f_{a-b}(0)$ it is straightforward that for $t'\geq 0$,
\begin{equation}
E^{*}[f_{X^{i}_{s}+X^*_{t'}}(x)]=E[f_{X^{i}_{s}}(-X^*_{t'}+x)]]=P^{t'}f_{X^{i}_{s}}(x),
\end{equation}
hence we obtain that
\begin{eqnarray*}
 E[F_L(i,i)] & = & 2E\left[\int_{L/2}^{L}\int_{0}^{L-s}\langle f_{X^{i}_{s}},GP^{t'}f_{X^{i}_{s}}\rangle dt'ds\right]\\
 & \stackrel{(\ref{Greendef})}{=} & 2E\left[\int_{L/2}^{L}\int_{0}^{L-s}\int_{t'}^{\infty}\langle f_{X^{i}_{s}},P^{u}f_{X^{i}_{s}}\rangle dudt'ds\right]\\
 & \stackrel{(\ref{eq:Ptfbound})}{\leq} & \int_{L/2}^{L}\int_{0}^{L-s}\int_{t'}^{\infty}\frac{c}{(1\vee u)^{d/2}}dudt'ds\\
 & \leq & \int_{L/2}^{L}\int_{0}^{L-s}c'(1\vee t')^{1-d/2}dt'ds\leq c''L.
\end{eqnarray*}
The claim (\ref{stupid1}) hence follows.

Now we prove (\ref{stupid2}). Similarly, we know that for all $i,j=1,\ldots,M$,
$i\neq j$,
\begin{eqnarray*}
E[F_L(i,j)]  & =  &E\left[\int_{[L/2,L]^{2}}\langle f_{X^{i}_{s}},Gf_{X^{j}_{t}}\rangle dtds\right]=E\left[\int_{[L/2,L]^{2}}\langle f_{X^{i}_{s}},GP^{t}f_{z_{j}}\rangle dtds\right]\\
 & \stackrel{(\ref{Greendef})}{=}  &E\left[\int_{[L/2,L]^{2}}\int_{t}^{\infty}\langle f_{X^{i}_{s}},P^{u}f_{z_{j}}\rangle dudtds\right]\\
 & \stackrel{(\ref{eq:Ptfbound})}{\underset{(L/2\geq1)}{\leq}} & \int_{[L/2,L]^{2}}\int_{t}^{\infty}\frac{C'}{u^{d/2}}dudtds\leq  CL^{3-d/2}.
\end{eqnarray*}
This confirms (\ref{stupid2}) as well as (\ref{bigupperbound1}) and finishes the proof of Lemma \ref{lem:greenfunctiondetailedcontrol}.
\end{proof}
\medskip
We now give a very brief introduction to Brownian capacity. We refer  readers to \cite{PortStone78} or Chapter 2 of \cite{Szni98a} for more details.

Let $K$ be a compact subset of $\mathbb{R}^{d}$. We denote by $e_{K}$
the equilibrium measure of $K$ (see Theorem 1.10, p.~58 of \cite{PortStone78}),  which is supported on the boundary of $K$. We denote by $\widetilde{e}_K$ the normalized equilibrium measure and denote  by $\mathrm{cap}(K)$ the (Brownian) capacity  of $K$ which is equal to the total mass of $e_K$.   There is a basic property relating equilibrium measure and the hitting probabilities which we will make repeated use of later in this work (see e.g.~the proof of Theorem 1.10, p.58 in \cite{PortStone78}):
\begin{equation}
P_{z}(H_K<\infty)=\int_{\mathbb{R}^d} g(z,y)e_{K}(dy).\label{eq:hitproba}
\end{equation}
The following lemma, which is a simple corollary of (\ref{eq:hitproba}),  is also a useful tool for estimating the Brownian capacity of a set.
\begin{lem}
For a compact $K\subset\mathbb{R}^{d}$ with positive volume, one has
\begin{equation}
\frac{{\rm Volume}(K)}{\sup_{z\in K}\int_{K}g(z,y)dy}\leq\mathrm{cap}(K)\leq\frac{{\rm Volume}(K)}{\inf_{z\in K}\int_{K}g(z,y)dy}.\label{eq:capest}
\end{equation}
\end{lem}
\begin{proof}
Note that for $z\in K$, $P_z(H_K<\infty)=1$. This implies that
$$
\int_{K}\int_{K} g(z,y) e_{K}(dy) dz={\rm Volume}(K).
$$
The claim (\ref{eq:capest}) follows from (\ref{eq:hitproba}) and the observation that $g(y,z)=g(z,y)$.
\end{proof}

The  Brownian capacity satisfies sub-additivity and monotonicity. It is also invariant under translations and rotations. More precisely, for compact $A,B\subset\mathbb{R}^{d}$, one has
\begin{equation}
{\rm cap}(A\cup B)\leq{\rm cap}(A)+{\rm cap}(B).\label{eq:union}
\end{equation}
And if $A\subset B$, then 
\begin{equation}
{\rm cap}(A)\leq{\rm cap}(B).\label{eq:monotonicity}
\end{equation}
For $x\in\mathbb{R}^d$ and $\vec{\rho}$ some rotation in $\mathbb{R}^d$, if $A'=A+x$ and $A''=\vec{\rho}(A)$ then
\begin{equation}
{\rm cap}(A')={\rm cap}(A)={\rm cap}(A'').
\end{equation}
See (4.15), Chap.~2, p.~70 and (4.17), Chap.~2, p.~71, in \cite{Szni98a} for more details.

It is a classical result (see e.g.~(3.55), p.~63 in \cite{Szni98a}) that for $R\geq0$
\begin{equation}\label{eq:ballcap}
{\rm cap}(B(0,R))=CR^{d-2}.
\end{equation}

We now state a classical variational characterization of the Brownian capacity (see Theorem 4.9, Chap 2, p. 76 in \cite{Szni98a}).
\begin{thm}\label{Thm4.9}
Given $K$, a compact subset of $\mathbb{R}^d$, we denote by $M_{1}(K)$ the space of probability
measures on $K$. Then
\begin{equation}
\mathrm{cap}(K)=\left(\inf\left\{{\cal E}(\mu),\mu\in M_{1}(K)\right\}\right)^{-1}\label{eq:capacityvariation}
\end{equation}
where for $\mu\in M_{1}(K)$, ${\cal E}(\mu)=\int_{K\times K}g(x,y)\mu(dx)\mu(dy)$.
\end{thm}

\subsection{Brownian interlacements}\label{ss:BIintro}

We now turn to the definition and basic properties of Brownian interlacements. The readers are referred to Section 2 of \cite{SznBI} for a complete description of the definition of this model.

We first remind readers the definition of the path space $W$ at the beginning of Section \ref{bmpt}. We consider $W^{*}$ the set of equivalence classes of trajectories
in $W$ modulo time-shift, i.e.,
\begin{equation}
W^{*}=W/\sim,
\end{equation}
where $w\sim w'$, if $w(\cdot)=w'(\cdot+t)$ for some $t\in\mathbb{R}$. Without loss of preciseness we still refer to elements of $W^*$ as ``trajectories''. We denote by $\pi^*$ the canonical projection on $W^*$ and introduce
the $\sigma$-algebra 
\begin{equation}
{\cal W}^{*}=\{A\subset W^{*};(\pi^{*})^{-1}(A)\in{\cal W}\},
\end{equation}
which is the largest $\sigma$-algebra on $W^{*}$ such that $(W,{\cal W})\stackrel{\pi^{*}}{\longrightarrow}(W^{*},{\cal W}^{*})$
is measurable.

Given a compact subset $K$ of $\IR^{d}$, we write $W_{K}$ for the subset of trajectories
of $W$ that enter $K$, and $W_{K}^{*}$ for its image under $\pi^{*}$. We can now introduce the measurable map $\pi_K$ from $W_{K}^{*}$ into $W_{+}$
defined by 
\begin{equation}\label{eq:pikdef}
\pi_K:w^{*}\in W_{K}^{*}\mapsto \big(w(H_{K}+t)\big)_{t\ge0},
\end{equation}
for any $w\in W_{K}$ such that $\pi^{*}(w)=w^{*}$. 

We now introduce the canonical space for the Brownian interlacement
point process, namely the space of point measures on $W^{*}\times\IR^{+}$,
\begin{align}\label{eq:standardcons}
\Omega=\Big\{ & \widetilde{\omega}=\sum_{i\ge0}\delta_{(w_{i}^{*},\alpha_{i})},\;\mbox{with}\;(w_{i}^{*},\alpha_{i})\in W^{*}\times[0,\infty)\;\mbox{and}\;\widetilde{\omega}(W_{K}^{*}\times[0,\alpha])<\infty, \\ 
 & \mbox{for any compact subset \ensuremath{K} of \ensuremath{\mathbb{R}^{d}} and \ensuremath{\alpha\ge0\Big\}}}.\nonumber 
\end{align}

By Theorem 2.2 and (2.22) in \cite{SznBI}, there exists a unique $\sigma$-finite measure $\nu$ on $(W^{*},{\cal W}^{*})$
such that for each compact subset $K$ of $\IR^{d}$, 
\begin{equation}
\mbox{the image of \ensuremath{1_{W_{K}^{*}}\nu} under \ensuremath{w^{*}\mapsto \pi_K(w^{*})} equals \ensuremath{P_{e_{K}}}}.\label{mapping}
\end{equation}

We endow $\Omega$ with the $\sigma$-algebra ${\cal A}$ generated
by the evaluation maps $\widetilde{\omega}\mapsto\widetilde{\omega}(B)$, for $B\in{\cal W}^{*}\otimes{\cal B}(\mathbb{R}^{+})$,
and denote by $\mathbb{P}$ the law on $(\Omega,{\cal A})$ of the
Poisson point measure with intensity measure $\nu\times d\alpha$
on $W^{*}\times\mathbb{R}^{+}$.

When $\widetilde{\omega}\in\Omega$, $\alpha\ge0$, $r\ge0$, we define  \textit{Brownian interlacements at level $\alpha$ with
radius $r$} through the following formula
\begin{equation}\label{eq:standardcons2}
\begin{array}{l}
\mbox{\ensuremath{{\cal I}_{r}^{\alpha}(\widetilde{\omega})=\bigcup\limits _{i\ge0:\alpha_{i}\le\alpha}\;\;\bigcup\limits _{s\in\mathbb{R}}B(w_{i}(s),r)}, where \ensuremath{\widetilde{\omega}=\sum_{i\ge0}\delta_{(w_{i}^{*},\alpha_{i})}}}\\
\mbox{and \ensuremath{\pi^{*}(w_{i})=w_{i}^{*}} for \ensuremath{i\ge0}}.
\end{array}
\end{equation}
By definition, $\mathcal{I}^{\alpha}_r$ is almost surely a closed
subset of $\IR^{d}$. 

We easily see that for $\alpha\geq\alpha'\geq 0$ and $r\geq r' \geq 0$, under the measure $\mathbb{P}$, $\mathcal{I}^\alpha_r$ is monotonously increasing with respect to both $\alpha$ and $r$, i.e.,
\begin{equation}\label{eq:monotonecoupling}
\mathcal{I}^\alpha_r\supseteq \mathcal{I}^{\alpha'}_r\qquad\mbox{ and }\qquad\mathcal{I}^\alpha_r\supseteq \mathcal{I}^{\alpha}_{r'}.
\end{equation}
We also immediately see that  for all $\alpha_1,\ldots,\alpha_n\geq0$, $r\geq0$ and independent Brownian interlacements $\mathcal{I}_i\stackrel{\rm law}{=} \mathcal{I}^{\alpha_i}_r$,
\begin{equation}\label{eq:addcouple}
\bigcup_{i=1}^{n}\mathcal{I}_i \stackrel{\rm law}{=} \mathcal{I}^{\alpha}_r\mbox{    where $\alpha=\alpha_1+\cdots+\alpha_n$}.
\end{equation}

We now give a local picture of Brownian interlacements. In fact, given $K$, a compact subset of $\IR^{d}$
and $\alpha\ge0$, the function from $\Omega$ to the set of finite point measures on $W_{+}$ 
\begin{equation}\label{eq:mukadef}
\mu_{K,\alpha}(\widetilde{\omega})=\sum_{i\ge0}1_{\{\alpha_{i}\le\alpha,w_{i}^{*}\in W_{K}^{*}\}}\,\delta_{\pi_K(w_{i}^{*})},\;\mbox{when}\;\;\widetilde{\omega}=\sum_{i\ge0}\delta_{(w_{i}^{*},\alpha_{i})}\in\Omega,
\end{equation}
satisfies, by (\ref{mapping}), that
\begin{equation}\label{eq:local}
\mbox{$\mu_{K,\alpha}$ is a Poisson point process on $W_+$ with intensity measure $\alpha P_{e_K}$.}
\end{equation}
It follows from (\ref{eq:local}) that we can give a simple characterization of the law of $\mathcal{I}_r^{\alpha}$ (see (2.32) in \cite{SznBI}):  
for all compact $K\subset\mathbb{R}^d$, 
\begin{equation}\label{simplechar}
\mathbb{P}[\mathcal{I}_{r}^{\alpha}\cap K=\emptyset]=e^{-\alpha\cdot\mathrm{cap}(B(K,r))}.
\end{equation}
We call the complement
of ${\cal I}_{r}^{\alpha}$, 
the \textit{vacant set} of Brownian interlacements:
\begin{equation}
{\cal V}_{r}^{\alpha}(\widetilde{\omega})=\mathbb{R}^{d}\backslash{\cal I}_{r}^{\alpha}(\widetilde{\omega}),\;\mbox{for}\;\widetilde{\omega}\in\Omega,\,\alpha>0,\, r\geq0.
\end{equation}
Note that ${\cal V}_r^{\alpha}$ is almost surely an open subset of $\IR^{d}$. Thanks to this, showing whether ${\cal V}_r^{\alpha}$ percolates is equivalent to showing whether it contains a continuous path to infinity, see the paragraph below (\ref{latticeobservation}). See also the second-to-last paragraph of Section \ref{se:basic}.

Now we recall some useful properties of Brownian interlacements. For all $\alpha>0$, $r>0$, $y\in\mathbb{R}^d$, $\lambda>0$, under $\mathbb{P}$, we know that (see (2.33), (2.35) and (2.36) in \cite{SznBI})
\begin{align}
\mathcal{I}^\alpha_r+y&\mbox{ has the same law as } \mathcal{I}^\alpha_r \;\;\;\,& \qquad\qquad\mbox{(translation invariance);}\label{translationinv}\\
\lambda\mathcal{I}^\alpha_r&\mbox{ has the same law as }\mathcal{I}^{\lambda^{2-d}\alpha}_{\lambda r}\;\, &\qquad\qquad\mbox{(scaling)};\label{scaling}\\
\mathcal{I}^\alpha_0 &\mbox{ is a.s.~connected, when }d=3&\qquad\qquad\mbox{(connectedness),} \nonumber\\
\;\;\mbox{ and}&\mbox{ a.s.~disconnected, when }d\geq 4&\qquad\qquad\mbox{(disconnectedness);} \label{connectedd3}\\
\mathcal{I}^\alpha_r&\mbox{ is rotational invariant}\qquad\;\;\quad&\qquad\qquad\mbox{(rotational invariance).}
\end{align}

We also regard $\mathcal{I}^\alpha_r$ itself as a random closed set in the space $(\Sigma,\sigma_f,Q^\alpha_r)$ where $\Sigma$ stands for the set of closed (and possibly empty) subsets of $\mathbb{R}^d$, endowed with $\sigma$-algebra $\sigma_f$, which is generated by the sets $\{F\in\Sigma;\; F\cap K=\emptyset\}$, where $K$ varies over the compact subsets of $\mathbb{R}^d$ (see Section 2.1, p.27 in \cite{Matheron}) and $Q^\alpha_r$ stands for its law. See below (2.31), \cite{SznBI} for more details. 

We end this subsection by the ergodicity of Brownian interlacements. 

\begin{prop}Let $(t_x)_{x\in\mathbb{R}^d}$ stand for the translations in $\mathbb{R}^d$. For all $\alpha,r\geq 0$,
\begin{equation}\label{eq:ergodicity}
\mbox{$(t_x)_{x\in\mathbb{R}^d}$ is a measure preserving flow on $(\Sigma,\sigma_f,Q^\alpha_r)$ which is ergodic.}
\end{equation}
Moreover,
 \begin{equation}\label{eq:01perco} 
\mathbb{P}[\mathcal{V}^\alpha_r\mbox{ percolates}]\in\{0,1\}.
\end{equation}
\end{prop}

\begin{proof}
We start with (\ref{eq:ergodicity}). It follows from (\ref{translationinv}) that $(t_x)_{x\in\mathbb{R}^d}$ is a measure preserving flow on $(\Sigma,\sigma_f,Q^\alpha_r)$. Ergodicity immediately follows if we prove that 
for any compact $K\in\mathbb{R}^d$, 
%
\begin{equation}
\label{eq:FFtau}\lim_{|x|\to \infty} \mathbb{E}[F(\mu_{K,\alpha}) \;F(\mu_{K,\alpha}) \circ \tau_x] = \mathbb{E}[F(\mu_{K,\alpha})]^2
\end{equation}
where $(\tau_x)_{x\in\mathbb{R}^d}$ stands for the translation on $\widetilde{\omega}\in\Omega$ by $-x$, i.e., if $\widetilde{\omega}=\sum_{i\geq0}\delta_{(w_i^*,\alpha_i)}\in\Omega$,
\begin{equation}
\tau_x \widetilde{\omega}=\sum_{i\geq0}\delta_{(w_i^*-x,\alpha_i)},\mbox{ for }x\in\mathbb{R}^d,
\end{equation}
for any $[0,1]$-valued measurable function $F$ on the set of finite point-measures on $W_+$, endowed with its canonical $\sigma$-algebra. By the translation invariance of Brownian interlacements, we can find $G$ (depending on $x$), with similar properties as $F$, such that the expectation in the left-hand side of (\ref{eq:FFtau}) equals $\mathbb{E}[F(\mu_{K,\alpha}) \,G(\mu_{K + x,\alpha})]$, while 
\begin{equation}\label{eq:GeF}
\mathbb{E}[G(\mu_{K + x,\alpha})]=\mathbb{E}[F(\mu_{K,\alpha})].
\end{equation}
By an argument similar to that between (2.11) and (2.15) in the proof of Theorem 2.1 in \cite{SznAnnal} we see that for $\alpha \ge 0$, $K$ compact, and $x\in\mathbb{R}^d$, $F,G$-measurable functions on the set of finite point measures on $W_+$ with values in $[0,1]$, when $|x|$ is sufficiently large (we assume ${\rm dist}(K,K+x)>0$),
\begin{equation}
\left|{\rm cov}_\mathbb{P} \big(F(\mu_{K,\alpha}), \,G(\mu_{K+x,\alpha})\big)\right| \le c\,\alpha \;\frac{{\rm cap}(K)^2}{{\rm dist}(K,K+x)}.
\end{equation}
This implies (\ref{eq:FFtau}) and thus concludes the proof of (\ref{eq:ergodicity}).

Since the event $\{\mathcal{V}^\alpha_r\mbox{ percolates}\}$ is translation invariant, (\ref{eq:01perco}) readily follows. 

\end{proof}

\subsection{Miscellaneous}
We start with some basic but useful facts on the renewal theory of Brownian motion. Let $(X_t)_{t\geq 0}$ be a $d$-dimensional Brownian motion. We define a sequence of stopping times $\tau^{N}$ inductively in the following way:
\begin{equation}
\tau^{1}=\inf\{s>0,|X_s-X_0|\geq1\}
\end{equation}
and when $N\geq1$
\begin{equation}
\tau^{N+1}=\tau^{1}\circ\theta_{\tau^{N}}+\tau^{N}.
\end{equation}
In other words, $\tau^{N}$ is the exit time after $\tau^{N-1}$ from
a ball of radius 1 centered at $X_{\tau^{N-1}}$.
For $t>0$, we write $N^{t}$ for the smallest integer $n$ such that $\tau^n$ is no less than $t$, i.e.,
\begin{equation}\label{Ntdef}
N^{t}=\min\{n\in\mathbb{N}; \;\tau^{n}\geq t\}.
\end{equation}
From standard renewal theory, see for example (3) in Section 4.1, p. 47 and (17) in Section 4.5, p. 58 of \cite{Renewal},  it is known that
\begin{equation}\label{NTm1}
E[N^{t}]\leq C't\quad\mbox{ and }\quad
{\rm Var}(N^{t})\leq C''t.
\end{equation}

We end this section by stating a generalized version of Borel-Cantelli lemma (see \cite{Borel} for more details). 
\begin{lem}\label{lem:Borel}
Consider a probability space $(\Omega,{\cal F},\mathbb{P})$
and a sequence of events $\Delta_{n}\in{\cal F}$. Let $\delta_{n}=1_{\Delta_{n}}$
be the indicator function of the event $\Delta_{n}$. If there exists
a sequence $b_{n}$ such that $\sum_{n}b_{n}=\infty$ and for any
$d_{i}\in\{0,1\}$, $i=1,\ldots,n-1,$
\begin{equation}
\mathbb{P}\left[\Delta_{n}\big|\delta_{1}=d_{1},\ldots,\delta_{n-1}=d_{n-1}\right]\geq b_{n}>0,
\end{equation}
then
\begin{equation}
\mathbb{P}[\limsup_{k}\Delta_{k}]=1.
\end{equation}

\end{lem}
\section{Graph distance between trajectories of Brownian interlacements}\label{se:2}
This section is entirely dedicated to the proof of Theorem \ref{graphdistance} on the graph distance of trajectories of Brownian interlacements. The lower and upper bounds are proved separately in Proposition \ref{prop:graphlb} of Section \ref{ss:lbgd}, and Proposition \ref{prop:gdub} of Section \ref{ss:ubgd}.

Throughout this section, we pick a fixed 
\begin{equation}
\alpha>0
\end{equation}
except for the proof of Theorem \ref{graphdistance}. From now on in this section we omit the dependence of constants on $\alpha$ in notation.
\subsection{The lower bound}\label{ss:lbgd}

In this subsection we prove in Proposition \ref{prop:graphlb} that for all $r>0$, almost surely ${\rm diam}(G_{\alpha,r})\geq s_{d}$, which constitutes the lower bound in Theorem \ref{graphdistance}.

We start with notations. For $l>0$, we write $S_{l}(x,y)\subset W^{*}$ for the collection of all trajectories that  intersects
both $B(x,l)$ and $B(y,l)$. 
 In the next lemma we show that $\nu(S_{r}(x,y))$ (recall that $\nu$ was defined in (\ref{mapping})) decays (as $|x-y|\to\infty$) at least as fast as $c|x-y|^{2-d}$. 
\begin{lem}
Let $x,y\in\mathbb{R}^d$. One has 
\begin{equation}\label{muSub}
\nu(S_{r}(x,y))\leq C(r)\min(|x-y|^{2-d},1).
\end{equation}
\end{lem}
From here onward we set
\begin{equation}\label{rchoiceBI}
\rho=\sqrt{d}/2+1,
\end{equation} and always apply this lemma with $r=2\rho$. Hence the constant from (\ref{muSub}) actually depends only on the dimension $d$.
\begin{proof}
Note that, if $w^*\in S_{r}(x,y)$, then it must either first pass through $B_1=B(x,r)$ and then pass through $B_2=B(y,r)$, or vice versa. By (\ref{mapping}), we obtain that
\begin{eqnarray*}
\nu(S_{2\rho}(x,y))& \leq & P_{ e_{B_1}}[H_{B_2}<\infty]+P_{ e_{B_2}}[H_{B_1}<\infty] \stackrel{\rm Symmetry}{=} 2 P_{e_{B_1}}[H_{B_2}<\infty]\\
 & \stackrel{(\ref{eq:hitproba})}{\leq} & \begin{cases}
C(r)& \textrm{ if }|x-y|<4r\\
2\int_{x\in \partial B_{1}}\int_{y\in \partial B_{2}}g(x,y)e_{B_{1}}(dx)e_{B_{2}}(dy) & \textrm{ otherwise.}
\end{cases}\\
 & \stackrel{(\ref{eq:ballcap})}{\leq}& \begin{cases}
C(r) & |x-y|<4r\\
C'(r) g(|x-y|-2r) & |x-y|\geq 4r
\end{cases}\\
 & \stackrel{(\ref{eq:greenvalue})}{\leq} & C''(r)\min(|x-y|^{2-d},1).
\end{eqnarray*}
This finishes the proof of (\ref{muSub}).
\end{proof}
Let $\widetilde\omega=\sum_{i\geq 0} \delta_{(w^*_i,\alpha_i)}$ be the interlacement process defined in Section \ref{ss:BIintro} and write $\omega=\sum_{i\geq 0, \alpha_i\leq \alpha}\delta_{w^*_i}$. By definition, $\omega$ has the law of a Poisson point process with intensity measure $\mu=\alpha\nu$. For $r>0$, we write
\begin{equation}
D_{r}(x,y)=\{\omega(S_{r}(x,y))\neq0\}
\end{equation}
for the event that there exists a trajectory in the intersection of the support of the interlacement process at level $\alpha$ and $S_{r}(x,y)$.
We then write 
\begin{equation}\label{Erho}
E_{r}=\{\mathrm{diam}(G_{\alpha,r})\leq s_d-1\}
\end{equation}
for the event that the diameter of $G_{\alpha,r}$ is no more than $s_d-1$.

In the next proposition we prove that, the probability that $x,y \in \mathcal{I}^\alpha_\rho$ when $E_\rho$ takes place, decays as $|x-y|$ tends to infinity. For convenience of argument we require that $x,y\in\mathbb{Z}^d$, which is sufficient for the proof by contradiction we will conduct later. 
\begin{prop}\label{prop:decay}
For all $x\neq y\in\mathbb{Z}^{d}$, one has
\begin{equation}\label{eq:probadecay}
\mathbb{P}[\{x,y\in\mathcal{I}_\rho^\alpha\}\cap E_{\rho}]\leq C|x-y|^{-1}.
\end{equation}
\end{prop}
\begin{proof}

On $\{x,y\in\mathcal{I}_\rho^\alpha\}\cap E_{\rho}$,
there exists $n\in\{1,\ldots,s_{d}-1\}$ with $\zeta_{0}=x$, $\zeta_{n+1}=y$
and $\zeta_{i}\in\mathbb{R}^{d}$ such that $D_{\rho}(\zeta_{i},\zeta_{i+1})$
happens for all $i=0,\ldots, n$ on different
trajectories in ${\rm Supp}(\omega)$, the support of $\omega$.  By (\ref{latticeobservation}), in this case there also exist $z_{0}=x$, $z_{n+1}=y$, $z_{i}\in\mathbb{Z}^{d}$, $i=1,\ldots,n$ with $z_i\neq z_{i+1}$, $i=0,\ldots,n$ and 
such that $D_{2\rho}(z_{i},z_{i+1})$, $i=0,\ldots,n$ happens on different
trajectories in ${\rm Supp}(\omega)$. We denote this event by $F_{z_1,\cdots,z_{n}}$ and by $\stackrel{*}{\sum}_n$ the sum over all $(n+1)$-tuples of pairwise different trajectories
$w^*_{0},\ldots,w^*_{n}\in{\rm Supp}(\omega)$. 
By definition of $F_{z_1,\cdots,z_{n}}$, we have 
\begin{equation}
\mathbb{P}[F_{z_1,\cdots,z_{n}}] \leq \mathbb{E}\Big[{\sum^{*}}_n\prod_{i=0}^{n}1_{w^*_{i}\in S_{2\rho}(z_{i},z_{i+1})}\Big]
\end{equation}
Writing $\#$ as a shorthand for $\{z_{1},\ldots z_{n}\in\mathbb{Z}^{d},z_{i}\neq z_{i+1}\mbox{ for }i=0,\ldots,n\}$, we have
\begin{eqnarray}
 \mathbb{P}\big[\{x,y\}\in\mathcal{I}_\rho^\alpha\}\cap E_{\rho}\big]& \leq & \sum_{n=1}^{s_{d}-1}\sum_{\#}\mathbb{P}[F_{z_1,\cdots,z_{n}}]\nonumber\\
 & \leq & \sum_{n=1}^{s_{d}-1}\sum_{\#}\mathbb{E}\Big[{\sum^{*}}_n\prod_{i=0}^{n}1_{w^*_{i}\in S_{2\rho}(z_{i},z_{i+1})}\Big]\\
 & \stackrel{(*)}{=} & \sum_{n=0}^{s_{d}-1}\sum_{\#}\prod_{i=0}^{n}\mathbb{E}\Big[\mu\big(S_{2\rho}(z_{i},z_{i+1})\big)\Big]\stackrel{\triangle}{=}{\rm I},\nonumber
\end{eqnarray}
where we obtain $(*)$ from Slivnyak-Mecke theorem, see Chapter 13.1, especially Proposition 13.1.VII, in \cite{DaVe}, see also the last paragraph of the proof of Lemma 3.1 in \cite{GDRI}. Moreover, by  (1.38) of Proposition 1.7 in \cite{Hara},
\begin{equation}\label{eq:convolution}
\sum_{\#}\prod_{i=0}^{n}\min\Big(1,|z_i-z_{i+1}|^{2-d}\Big)\left\{\begin{array}{ll} \leq C(n)|x-y|^{2n+2-d} &\mbox{ if $n<s_d$,}\\ =\infty &\mbox{ otherwise.}\end{array}\right.
\end{equation}
We hence obtain that
\begin{equation}
\begin{split}
{\rm I}\;\stackrel{(\ref{muSub})}{\leq}& \sum_{n=1}^{s_d-1}\sum_{\#}\prod_{i=0}^{n}c\min\Big(1,|z_i-z_{i+1}|^{2-d}\Big)\\
\stackrel{(\ref{eq:convolution})}{\leq}&\sum_{n=1}^{s_d-1} C(n)|x-y|^{2n+2-d} \leq C|x-y|^{2s_{d}-d}\stackrel{(\ref{eq:sddef})}{\leq} C|x-y|^{-1}.
\end{split}
\end{equation}
This ends the proof of (\ref{eq:probadecay}). 
\end{proof}
Now we rephrase and prove the main claim of this subsection.

\begin{prop}\label{prop:graphlb}
For $d\geq 3$, $\alpha>0$ and $r>0$ one has
\begin{equation}\label{graphlb}
\mathbb{P}[{\rm diam}(G_{\alpha,r})\geq s_{d}]=1.
\end{equation}
\end{prop}
\begin{proof}
It is a simple fact that ${\rm diam}(G_{\alpha,r})\geq1$ because $G_{\alpha,1}$ has more than one vertex. When $d=3,4$, $s_{d}=1$, hence the claim (\ref{graphlb}) follows directly. Therefore it suffices to prove (\ref{graphlb}) for $d\geq5$. 

Thanks to the scaling property of Brownian interlacements (see (\ref{scaling})),
we can assume without loss of generality that $r=1$.

We assume by contradiction that for some $\delta>0$ (recall the definition of $E_1$ in (\ref{Erho})),
\begin{equation}\label{contradictBI}
\mathbb{P}[E_{1}]\geq\delta.
\end{equation}

On the one hand, by (\ref{simplechar}), we can take large $R$ such that
\begin{equation}\label{pickbigball}
\mathbb{P}\big[\mathcal{I}_{1}^{\alpha}\cap B(0,R)\neq\emptyset\big]\geq1-\delta/3.
\end{equation}
By the translation invariance of Brownian interlacements, (\ref{pickbigball}) also holds if one replaces $0$ by any $x\in\mathbb{R}$, hence with the assumption (\ref{contradictBI}), we obtain that uniformly for any $x\in\mathbb{R}^d$,
\begin{equation}\label{contra}
\mathbb{P}\Big[\big\{\mathcal{I}_{1}^{\alpha}\cap B(0,R)\neq\emptyset\big\}\cap\big\{\mathcal{I}_{1}^{\alpha}\cap B(x,R)\neq\emptyset \big\}\cap E_{1}^{}\Big]\geq\delta/3.
\end{equation}

On the other hand, we now show that
\begin{equation}\label{contraBI2}
\lim_{|x|\to\infty}\mathbb{P}\Big[\big\{\mathcal{I}_{1}^{\alpha}\cap B(0,R)\neq\emptyset\big\}\cap\big\{\mathcal{I}_{1}^{\alpha}\cap B(x,R)\neq\emptyset \big\}\cap E_{1}^{}\Big]=0.
\end{equation}
By (\ref{latticeobservation}), we know that with our choice of $\rho$ (see (\ref{rchoiceBI})), 
\begin{equation}
\big\{\mathcal{I}_{1}^{\alpha}\cap B(0,R)\neq\emptyset\big\}\bigcap\big\{\mathcal{I}_{1}^{\alpha}\cap B(x,R)\neq\emptyset \big\}\subset\bigcup_{y\in B(0,R)\cap\mathbb{Z}^{d}}\bigcup_{z\in B(x,R)\cap\mathbb{Z}^{d}}\big\{y,z \in {\cal I}_{\rho}^{\alpha}\big\}.
\end{equation}
Hence, by (\ref{eq:probadecay}) and the fact that $E_1 \subseteq E_\rho$, we obtain that
\begin{equation}\label{contraBIlb}
\begin{split}
&\quad\mathbb{P}\Big[\big\{\mathcal{I}_{1}^{\alpha}\cap B(R)\neq\emptyset\big\}\cap\big\{\mathcal{I}_{1}^{\alpha}\cap B(x,R)\neq\emptyset \big\}\cap E_{1}^{}\Big]\\
\leq&\sum_{y\in B(0,R)\cap\mathbb{Z}^{d}}\sum_{z\in B(x,R)\cap\mathbb{Z}^{d}}\mathbb{P}\Big[\big\{y,z \in {\cal I}_{\rho}^{\alpha}\big\}\cap E_{\rho}\Big]\stackrel{(\ref{eq:probadecay})}{\leq} CR^{2d}|x|^{-1}.\end{split}
\end{equation}

The right-most term in (\ref{contraBIlb}) converges to $0$ as $x\to\infty$. This implies (\ref{contraBI2}), which creates a contradiction with (\ref{contra}), concluding the proof of (\ref{graphlb}).
\end{proof}

\subsection{Some preparatory capacity estimates for the upper bound}
This subsection is dedicated to some preliminary results 
for the proof of the upper bound in Theorem \ref{graphdistance}. 
The central result in this subsection is Proposition \ref{annuluscaplb}. 
As noted in Proposition \ref{34notneeded}, the cases $d=3$ and $d=4$ are classical results, hence throughout this subsection we will always assume  $d\geq 5$.

We start with the following basic property of Brownian motion which follows from integrating (\ref{transition})  the transition density of Brownian motion.
\begin{lem}\label{lem:exitproba}
There exists $c_{0}\in\mathbb{R}^{+}$, such that for all $R>0$ 
\begin{equation}\label{eq:exitproba}
P_{0}[T_{B^\circ(0,R/2)}>c_{0}R^{2}]\geq0.99.
\end{equation}
\end{lem}

Now we define the central object of this subsection.
\begin{definition}\label{df:cactus}
 Let $R>1$ be a positive real number. Let $w_1,\ldots,w_N$
be a series of $N$ trajectories in $W_+$, with $w_i(0)=x_{i}\in\mathbb{R}^{d}$.
We denote by $\overline{W}_{N}=(w_1,\ldots,w_N)$ the collection of these trajectories. We denote by
\begin{equation}\label{eq:cactusdef}
\Phi(\overline{W}_{N},R)=\cup_{i=1}^{N}B\Big(w_i\big([0,T_{B^\circ(x_{i},R/2)}\wedge c_{0}R^{2}],1\big)\Big)
\end{equation}
the union of the sausages of these trajectories stopped at the
smaller of $c_{0}R^{2}$ and the exiting time of the ball of radius
$R/2$ centered at the respective starting point. 
\end{definition}
 Let $X^{(1)},\ldots,X^{(N)}$
be $N$ independent Brownian motions, with $X^{(i)}_0=x_i$. Let $\overline{X}_N=(X^{(1)},\ldots,X^{(N)})$. In the rest of this section we are going to study the capacity
of $\Phi(\overline{X}_{N},R)$. 

We start with the upper bounds on its first
and second moments.
\begin{lem}
\label{lem:Ecactusub}
Let $\overline{X}_N$ be defined as above. We denote its joint law by $E$. One has 
\begin{equation}
E[\mathrm{cap}(\Phi(\overline{X}_{N},R))]\leq CNR^{2},\label{eq:detercacub}
\end{equation}
and
\begin{equation}
E[\mathrm{cap}(\Phi(\overline{X}_{N},R))^{2}]\leq CN^{2}R^{4}.\label{eq:detercalc2ndub}
\end{equation}
\end{lem}
\begin{proof}
By the definition of $N^t$ (See (\ref{Ntdef})), 
 we know that for each $i=1,\ldots,N$, 
\begin{equation}
\Phi((X^{(i)}),R)\textrm{ is covered by no more than }N^{c_{0}R^{2}}\textrm{ balls of radius }2,\,P_{x_i}\mbox{-almost surely.}\label{eq:covercactusbyballs}
\end{equation}
Thanks to the independence of $X^{(i)}$, $i=1,\ldots,N$, the sub-additivity of Brownian capacity, see (\ref{eq:union})), and in the case of (\ref{eq:detercalc2ndub}) also by discrete H\"older inequality,
 to prove (\ref{eq:detercacub}) and (\ref{eq:detercalc2ndub}) it suffices to verify that 
\begin{equation}
E_{x_i}[N^{c_{0}R^{2}}]\leq cR^{2}\textrm{ and }E_{x_i}[(N^{c_{0}R^{2}})^2]\leq c'R^{4}.\label{eq:verify}
\end{equation}
In fact, one can easily check (\ref{eq:verify}) by (\ref{NTm1}). This finishes the proof of (\ref{eq:detercacub}) and (\ref{eq:detercalc2ndub}).
\end{proof}

In the next Proposition we derive a lower bound on the first moment of ${\rm cap}(\Phi(\overline{X}_{N},R))$.
\begin{prop}\label{prop:CactusEmin}
With the same setup as in Lemma \ref{lem:Ecactusub}, one has 
\begin{equation}\label{eq:CactusEmin}
E[\mathrm{cap}(\Phi(\overline{X}_N,R)]\geq C\min(NR^{2},R^{d-2}).
\end{equation}
\end{prop}
The key method used in the proof of this proposition is Theorem \ref{Thm4.9}.
\begin{proof}
In this proof, we use superscripts to distinguish stopping times with respect to different Brownian motions, i.e., $T^{(i)}_\cdot$ stands for the stopping times with respect to $X^{(i)}$.
Write
\begin{equation}
J=\big\{ i\in\{1,\ldots,N\},T^{(i)}_{B(x_i,R/2)}\geq c_{0}R^{2}/2\big\}.
\end{equation}
See (\ref{eq:exitproba}) for the definition of $c_0$ in Lemma \ref{lem:exitproba}.

We then define a probability measure $m$ on $\Phi(\overline{X}_{N},R)$
through the density function $h_m$ (recall the definition of $f_z(\cdot)$ in the statement of Lemma \ref{lem:greenfunctiondetailedcontrol}) 
\begin{equation}
h_m(x)=\begin{cases}
c\big(|J|R^{2}\big)^{-1}\sum_{i\in J}\int_{c_0 R^{2}/4}^{c_0R^{2}/2}f_{X^{(i)}_t}(x)dt & \textrm{ if }J\neq\emptyset;\\
\big(\mathrm{Volume}(\Phi(\overline{X}_{N},R))\big)^{-1} & \textrm{ if }J=\emptyset,
\end{cases}
\end{equation}
where $c$ is a constant that makes $m$ a probability measure.
Let $A=\{|J|\geq\frac{1}{2}N\}$. By (\ref{eq:exitproba}), 
we know that 
\begin{equation}\label{eq:Adef}
P[A]\geq c_1
\end{equation}
for some $c_1>0$.

Since $m\in M_{1}(\Phi(\overline{X}_{N},R))$, by Theorem \ref{Thm4.9}, we see that (recall the definition of ${\cal E}(m)$ below (\ref{eq:capacityvariation}))
\begin{equation}\label{Cac1}
E\Big[\mathrm{cap}\big(\Phi(\overline{X}_{N},R)\big)\Big] \geq E\big[{\cal E}(m)^{-1}\big]\geq E\big[{\cal E}(m)^{-1};A\big].
\end{equation}
And on $A$ (note that in this case $J\neq\emptyset$)
 we obtain that (recall the definition of $F_L(i,j)$ in (\ref{eq:FLdef}))
\begin{equation}
{\cal E}(m)=\frac{C}{|J|^{2}R^{4}}\sum_{i,j\in J}F_{\frac{c_0}{2}R^2}(i,j).
\end{equation}
Therefore by Cauchy-Schwarz inequality, it follows that 
\begin{equation}
\begin{split}
E\big[{\cal E}(m)^{-1};A\big]\;\;  \geq \;\;& E\left[\frac{C}{|J|^{2}R^{4}}\sum_{i,j\in J}F_{\frac{c_0}{2}R^2}(i,j);A\right]^{-1}\cdot P[A]^{2}\\
  \stackrel{(\ref{eq:Adef})}{\geq} & C'N^{2}R^{4}E\left[\sum_{i,j\in J}F_{\frac{c_0}{2}R^2}(i,j);\, A\right]^{-1}\geq C'N^{2}R^{4}E\left[\sum_{i,j=1,\ldots,N}F_{\frac{c_0}{2}R^2}(i,j)\right]^{-1} \label{Cac2}\\
  \stackrel{(\ref{bigupperbound1})}{\geq} & C''N^{2}R^{4}\Big/\big(c''(NR^{2}+N^{2}R^{6-d})\big) \geq  C'''\min(NR^{2},R^{d-2}).
\end{split}
\end{equation}
The claim (\ref{eq:CactusEmin}) hence follows by putting (\ref{Cac1}) and (\ref{Cac2}) together. 
\end{proof}
Now we turn to the construction of the central object of this and the next subsection. Let $A$ be a compact
set 
and let $\omega$ be a point measure on $W^*$ such that
$
\omega=\sum_{i\geq0}\delta_{w_{i}}.
$
Let $N=\omega(W^*_A)$,
we denote by (recall the definition of $\pi_A$ in (\ref{eq:pikdef})) \begin{equation}
\overline{W}(\omega,A) =(\pi_{A}(w^*_{i_1}),\ldots\pi_{A}(w^*_{i_N})) \textrm{, where }
\{i_1,\ldots,i_N\}=\{i\geq 0;\,w^*_i\in W^*_A \}
\end{equation}
the collections of trajectories in the support of  $\omega$ that touch $A$. We then write
\begin{equation}
\Psi(\omega,A,R)=\Phi(\overline{W}(\omega,A),R).
\end{equation}

Now we introduce some notation on the restriction of point measures on $W^*$.
\begin{definition}\label{decompdef}
Let $r,R$ be real numbers such that $1<r<R$, and let $\omega$ be a point measure on $W^{*}$ such that
$\omega=\sum_{i\geq0}\delta_{w^*_{i}}$ with $w^*_{i}\in W^{*}$. We write
\begin{equation}
\omega_{r}=\sum_{i\geq0,\;w^*_i\in W^*_{B(r)}}\delta_{w^*_{i}}
\end{equation}
for  a new point measure which is the restriction of $\omega$ to the set of trajectories
that intersect $B(r)$,  write
\begin{equation}
\omega_{r,\infty}=\omega-\omega_{r}
\end{equation}
for the restriction of $\omega$ to the set of trajectories that do
not intersect $B(r)$, and write
\begin{equation}
\omega_{r,R}=\omega_{R}-\omega_{r}
\end{equation}
for the restriction of $\omega$ to the set of trajectories that intersect
$B(R)$ but do not intersect $B(r)$.
\end{definition}

We now let $\widetilde\omega=\sum_{i\geq 0}\delta_{(w^*_i,\alpha_i)}$ be the interlacement process and apply 
\begin{equation}\label{eq:omegadef}
\omega=\sum_{i\geq 0,\alpha_i\leq\alpha}\delta_{w^*_i}
\end{equation}
to Definition \ref{decompdef} above. 
Note that point processes $\omega_{r}$ and $\omega_{r,\infty}$ are independent from each other.

We now give an estimate on the capacity of a random
cactus constructed from $\omega_{r,\infty}$ and a compact subset of
$\mathbb{R}^{d}$. 
\begin{prop}\label{annuluscaplb}
Let $A$ be a compact subset of $\mathbb{R}^{d}$, and let $r,R$ be real numbers such that $1<r<R$, one has 
\begin{equation}\label{eq:annuluscaplb}
\mathbb{E}[{\rm cap}(\Psi(\omega_{r,\infty},A,R))]\geq c\min({\rm cap}(A)R^{2},R^{d-2})-C r^{d-2}R^{2}.
\end{equation}
\end{prop}
To prove Proposition \ref{annuluscaplb} we need the following lemma.  We omit its proof because it is identical  to that of Lemma 4.3 in \cite{GDRI}, with
(\ref{eq:CactusEmin}) playing the role of (4.5) in Lemma 4.2 in \cite{GDRI}. 
\begin{lem}
For all $A$ compact subset of $\mathbb{R}^d$, one has
\begin{equation}
\mathbb{E}\Big[{\rm cap}\big(\Psi(\omega,A,R)\big)\Big]\geq c\min({\rm cap}(A)R^{2},R^{d-2}).\label{eq:Ecapmin}
\end{equation}
\end{lem}
\begin{proof}[Proof of Proposition \ref{annuluscaplb}]
We start with decomposing $\omega$:
\begin{equation}
\omega=\omega_{r}+\omega_{r,\infty}.
\end{equation}
By the definition of $\Psi$,
\begin{equation}
\Psi(\omega,A,R)=\Psi(\omega_{r,\infty},A,R)\cup\Psi(\omega_{r},A,R).
\end{equation}
By the sub-additivity of capacities, we obtain that
\begin{equation}
\mathbb{E}[{\rm cap}(\Psi(\omega_{r,\infty},A,R))]\geq \mathbb{E}[{\rm cap}(\Psi(\omega,A,R))]-\mathbb{E}[{\rm cap}(\Psi(\omega_{r},A,R))].
\end{equation}
Thanks to (\ref{eq:Ecapmin}), to prove (\ref{eq:annuluscaplb}), it suffices to show that 
\begin{equation}\label{eq:midub}
\mathbb{E}[{\rm cap}(\Psi(\omega_{r},A,R))]\leq C R^{2}r^{d-2}.
\end{equation}
To this end, suppose that
\begin{equation}
\omega_{r}=\sum_{i=1}^{N}\delta_{w^*_{i}}
\end{equation}
where $N=\omega(W^*_{B(r)})=\omega_r(W^*_{B(r)})$, $w^*_1,\ldots,w^*_N\in W^*$. 
By (\ref{eq:local})  and the definition of $\omega$ (see (\ref{eq:omegadef})) we know that $N\sim{\rm Pois}(\alpha{\rm cap}(B(r)))$ and conditioned on $N$ and the starting points $X^{(i)}_{0}$, $X^{(i)}$'s are independent Brownian motions.
Hence from (\ref{eq:detercacub}) we obtain that
\begin{equation}
\begin{split}
&\;\mathbb{E}\big[{\rm cap}(\Psi(\omega_{r},A,R)\big]\\
 \leq \;\;&\;\mathbb{E}\bigg[\sum_{n=0}^\infty \sup_{x_1,\ldots,x_n\in \partial A} \mathbb{E}\Big[\Phi(\overline{w},R)|N=n,\;\big(w_1(0),\ldots,w_n(0)\big)=\big(x_1,\ldots,x_n\big)\Big]\mathbb{P}[N=n]\bigg]\\ 
\overset{(\ref{eq:detercacub})}{\leq} &\; \mathbb{E}[N]cR^2\leq C R^{2}r^{d-2}.\end{split}
\end{equation}
This confirms (\ref{eq:midub}) as well as (\ref{eq:annuluscaplb}).
\end{proof}

\subsection{The upper bound}\label{ss:ubgd}

In this subsection, we prove the more difficult part of Theorem \ref{graphdistance}, namely the upper bound on the diameter of $G_{\alpha,1}$. As mentioned in Section \ref{sec:intro}, the proof scheme is parallel to that of \cite{GDRI} except a few simplifications. Let us also point out some technical issues we encounter when dealing with Brownian motion. First, the energy estimate (see Proposition \ref{prop:CactusEmin}) that gives the lower bound of the expected capacity of the truncated sausages requires a version  of Lemma 4.1 of \cite{GDRI} for Brownian motion. This is archived in Lemmas \ref{lem:Ptfbound} and \ref{lem:greenfunctiondetailedcontrol}. Second, when estimating capacities of the upper bound, we need an estimate on number of balls needed to cover a Brownian motion path, see (\ref{eq:covercactusbyballs}). 

We start by rephrasing our main goal in this subsection.
\begin{prop}\label{prop:gdub} For all $d\geq 3$, $\alpha>0$ 
\begin{equation}
\mathbb{P}[{\rm diam}(G_{\alpha,1})\leq s_{d}]=1.\label{eq:upper bound}
\end{equation}
\end{prop}
We first show that we only need to prove (\ref{eq:upper bound}) for
$d\geq5$.
\begin{prop}\label{34notneeded}
The claim (\ref{eq:upper bound}) is true for $d=3$ and $d=4$.\end{prop}
\begin{proof}
When $d=3$ or $d=4$, $s_{d}=1$, hence it suffices to show
that two Wiener sausages, regardless of their starting points, will
hit each other almost surely. The case $d=3$ is classical. The case
$d=4$ follows by Theorem 6.2 of \cite{Sausage4d}.
\end{proof}
From now on in this subsection we only consider $d\ge5.$
Let $(X_t)_{t\geq0}$ be a Brownian motion in $\mathbb{R}^{d}$ with $X_0=x$. We denote its law and the respective expectation by $P^{X}_{x}$ and $E^{X}_x$.
Let $\omega^{(2)},\omega^{(3)},\ldots\in\overline{\Omega}$ be i.i.d.~ interlacement processes 
at level $\alpha$, which are also independent
from $(X_t)_{t\geq 0}$. We denote by $\mathbb{P}^{(2)},\mathbb{P}^{(3)},\ldots$ their laws and by
$\mathbb{E}^{(2)},\mathbb{E}^{(3)},\ldots$ the respective expectations. For $s\geq1$,
we write $\mathbb{P}_{x}^{(s)}$ for the joint law $P^{X}_{x}\otimes \mathbb{P}^{(2)}\otimes \mathbb{P}^{(3)}\otimes\cdots\otimes \mathbb{P}^{(s)}$,
and $\mathbb{E}_{x}^{(s)}$ for the respective expectation. 

Now let $r$ and $R$ be positive reals such that $1<r<R$ and $|x|<R$. We define a sequence of random subsets of $\mathbb{R}^{d}$
associated to $X$ and $\omega^{(i)}$ in the following inductive manner.
We write 
\begin{equation}\label{eq:A1rRdef}
A^{1}(r,R)=\Phi(X(\cdot+T_{B^\circ(r)}),R),
\end{equation}
and for $s\geq2$, we build
$A^{s}(r,R)$ upon the measure $\omega_{r,\infty}^{(s)}$ (recall the
notation in Definition \ref{decompdef}) and $A^{s-1}(r,R)$ in the following manner:
\begin{equation}
A^{s}(r,R)=\Phi\left(\omega_{r,\infty}^{(s)},A^{s-1}(r,R),R\right).
\end{equation}
It is worth noting that 
\begin{equation}
\label{eq:inde}\mbox{$A^{s-1}(r,R)$ is independent from $\omega^{(s)}$.}
\end{equation}

\begin{rem}
One can inductively prove that 
\begin{equation}\label{inductionA}
A^{s-1}(r,R)\subset B(sR)
\end{equation} 
and this implies that the definition of $A^s(r,R)$ is not changed if we replace $\omega_{r,\infty}^{(s)}$ by $\omega_{r,sR}^{(s)}$, i.e., 
\begin{equation}\label{inductionAs}
A^{s}(r,R)=\Phi\left(\omega_{r,sR}^{(s)},A^{s-1}(r,R),R\right).  
\end{equation}


\end{rem}

Now we derive an upper bound on the second moment of ${\rm cap}(A^{s}(r,R))$.
\begin{prop}
\label{prop:cactus2ndub}For all $s\leq s_d$, one has
\begin{equation}
\mathbb{E}_{x}^{(s)}\left[{\rm cap}(A^{s}(r,R))^{2}\right]\leq CR^{2\min(d-2,2s)}.\label{eq:cactus2ndub}
\end{equation}
\end{prop}
Notice that the constants in the statement above and the proof below do depend on $s$. However, since we only look at $s\leq s_d$, we can safely drop this dependence in the notation.
\begin{proof}
In this proof we write $A^{s}$ as a shorthand for $A^{s}(r,R)$.  By the monotonicity
of the capacity 
and  (\ref{inductionA}), we know that 
\begin{equation}
{\rm cap}(A^{s})\leq{\rm cap}(B((s+1)R))\leq cR^{d-2}.
\end{equation}
Hence to prove (\ref{eq:cactus2ndub}), it suffices to show that for all $1\leq s\leq s_d$, one has 
\begin{equation}
\mathbb{E}_{x}^{(s)}\left[{\rm cap}(A^{s})^{2}\right]\leq c''R^{4s}.\label{eq:cactusinductionub}
\end{equation}
We now prove (\ref{eq:cactusinductionub}) by induction. When $s=1$, Lemma
\ref{lem:Ecactusub} implies that
\begin{equation}
\mathbb{E}_{x}^{(s)}\left[({\rm cap}(A^{1}))^{2}\right]\leq c''R^{4}.
\end{equation}
Suppose now that (\ref{eq:cactusinductionub}) is true for some $s-1\geq 1$.
By the definition of $\Psi$, $A^{s}$ consisits of $\omega^{(s)}_{r,\infty}(W^*_{A^{s-1}})$ Wiener sausages. Moreover, we know that  conditioned on $A^{s-1}$ 
\begin{equation}\label{eq:leqPoisson}
\omega^{(s)}_{r,\infty}(W^*_{A^{s-1}})\leq \omega^{(s)}(W^*_{A^{s-1}})\sim {\rm Pois}\left(\alpha{\rm cap}(A^{s-1})\right).
\end{equation}
Hence, 
\begin{equation}
\begin{split}
\mathbb{E}_{x}^{(s)}\left[({\rm cap}(A^{s}))^{2}\right]&\stackrel{(\ref{eq:detercalc2ndub})}{\leq}cR^{4}\mathbb{E}_{x}^{(s)}\left[\omega^{(s)}_{r,\infty}(W^*_{A^{s-1}})^{2}\right]\stackrel{(\ref{eq:leqPoisson})}{\leq} cR^4\mathbb{E}_{x}^{(s)}\left[\omega^{(s)}(W^*_{A^{s-1}})^{2}\right]\\
&\stackrel{(\ref{eq:leqPoisson})}{\leq}cR^4\left(\mathbb{E}_{x}^{(s-1)}\left[(\alpha{\rm cap}(A^{s-1}))^{2}\right]+\mathbb{E}_{x}^{(s-1)}\left[\alpha{\rm cap}(A^{s-1})\right]\right)\\
& \stackrel{(\ref{eq:cactusinductionub})}{\underset{\rm Jensen}{\leq}} c'R^4 \left(R^{4s-4}+R^{2s-2}\right)\leq c''R^{4s}.
\end{split}
\end{equation}
The claim 
(\ref{eq:cactusinductionub}) thus follows thus by induction. This finishes the proof of (\ref{eq:cactus2ndub}).
\end{proof}

We now inductively derive a lower bound on the expectation of the capacity of $A^{s}(r,R)$. 
\begin{prop}
\label{prop:cactuslb}There exists an $\overline{\epsilon}\in(0,1)$,
such that for all $1\leq s\leq s_{d}$, and for all positive reals $r$ and $R$
that satisfy 
\begin{equation}
1<r^{d-2}\leq\overline{\epsilon} R,\label{eq:rRcondition}
\end{equation}
 one has 
\begin{equation}
\mathbb{E}_{x}^{(s)}\left[{\rm cap}(A^{s}(r,R))\right]\geq c(s)R^{\min(d-2,2s)}.\label{eq:Ecapinduction}
\end{equation}
for a sequence of positive constants $c(1),\ldots,c(s)$. In particular,
\begin{equation}
\mathbb{E}_{x}^{(s_{d})}[{\rm cap}(A^{s_{d}}(r,R))]\geq cR^{d-2}.\label{eq:cactussdinf}
\end{equation}

\end{prop}

\begin{proof}
As in the proof of Proposition \ref{prop:cactus2ndub}, we still write $A^{s}$ for $A^{s}(r,R)$. We postpone our choice
of $\epsilon$ until the end of this proof. We prove (\ref{eq:Ecapinduction})
by induction on $s$. By (\ref{eq:CactusEmin}) we know that 
\begin{equation}
\mathbb{E}_{x}^{(s)}[{\rm cap}(A^{1})]\geq c(1)R^{2}.
\end{equation}
Set $\epsilon(1)=1/2$. 
Let $2\leq s\leq s_d$, and assume the induction hypothesis holds for $s-1$:
\begin{equation}
\mathbb{E}_{x}^{(s-1)}[{\rm cap}(A^{s-1})]\geq c(s-1)R^{\min(d-2,2s-2)}\label{eq:caplbinduction}
\end{equation}
for some $c(s-1)>0$, for all $r,R$ such that $1<r^{d-2}\leq \epsilon(s-1)R$. By (\ref{eq:caplbinduction}) above and
the upper bound on its second moment (see (\ref{eq:cactus2ndub})), the Paley-Zygmund inequality (see e.g.~\cite{PaleyZ})
implies that there exists a positive constant $c_1(s-1)$ and $c_2(s-1)$ such that 
\begin{equation}\label{eq:inductionstep2}
\mathbb{P}_{x}^{(s-1)}[{\rm cap}(A^{s-1})\geq c_1(s-1)R^{\min(d-2,2s-2)}]\geq c_2(s-1).
\end{equation}
Then it follows from Proposition \ref{annuluscaplb} and the definition
of $\Psi$ that 
\begin{equation}\begin{split}
&\qquad\;\;\,\mathbb{E}_{x}^{(s)}[{\rm cap}(A^{s})]  \stackrel{(\ref{eq:annuluscaplb})}{\geq}  \mathbb{E}_{x}^{(s-1)}\left[c\min({\rm cap}(A^{s-1})R^{2},R^{d-2})-C r^{d-2}R^{2}\right]\\
&\;\;\,\geq\;\;\mathbb{E}_{x}^{(s-1)}\left[c'\min( R^{\min(d-2,2s-2)}\times R^{2},R^{d-2})1_{{\rm cap}(A^{s-1})\geq c_1(s-1)R^{\min(d-2,2s-2)}}\right]-C r^{d-2}R^{2}\\
 & \stackrel{(\ref{eq:inductionstep2})}{\geq}  c''(s-1)\min(R^{2}\times R^{\min(d-2,2s-2)},R^{d-2})-C r^{d-2}R^{2}\\
 & \stackrel{(\ref{eq:rRcondition})}{\geq} c''(s-1) R^{\min(d-2,2s)}-c'\epsilon' R^{3}\stackrel{(*)}{\geq}\frac{c''(s-1)}{2}R^{\min(d-2,2s)},
\end{split}
\end{equation}
with a choice of sufficiently small $\epsilon'$ which makes inequality
marked with $(*)$ valid.
Letting
$c(s)=c''(s-1)/2$ and $\epsilon(s)=\min(\epsilon(s-1),\epsilon')$,
we thus confirm the induction step for $s$. Hence (\ref{eq:Ecapinduction}) is verified for all $s=1,\ldots,s_d$. Let $\overline{\epsilon}=\epsilon(s_d)$. The claim (\ref{eq:cactussdinf}) follows from the definition of $s_d$ (see (\ref{eq:sddef})). \end{proof}
\begin{rem}
1) Note that in the proof above, in order to proceed with induction, one needs to have a bound like (\ref{eq:inductionstep2}) for ${\rm cap}(A^s)$. This is obtained through Paley-Zygmund inequality, for which the lower bound on its first moment and  the upper bound on the second moment are prerequisites.

\noindent 2) The combination of Propositions \ref{prop:cactus2ndub} and \ref{prop:cactuslb}
indicates that the $cR^{\min(d-2,2s)}$ is the right order for $\mathbb{E}_{x}^{(s)}[{\rm cap}(A^{s}(r,R))]$.	 Most importantly, $A^{s_d}$, a subset of $B((s_d+1)R)$, has capacity of order $R^{d-2}$. This means, in terms of capacity, $A^{s_d}$ ``saturates'' the ball $B((s_d+1)R)$. 
\end{rem}
One can use the proposition above to study the probability of another independent Brownian motion hitting $A^{s}(r,R)$.
\begin{definition}
\label{Zdef}Consider $d\geq5$. Let $(Z_t)_{t\geq 0}$ be a Brownian Motion on
$\mathbb{R}^{d}$ with $Z_{0}=z\in B(R)$,  independent from
$X$ and $\omega^{(2)},\ldots,\omega^{(s)}$. In the following part of this section, we denote its law by $P^{Z}_{z}$ and the respective expectation by $E^{Z}_{z}$. We write $\widetilde{\mathbb{P}}$ as a shorthand for $P^Z_{z}\otimes\mathbb{P}_{x}^{(s_{d})}$.
\end{definition}
\begin{prop} ($d\geq 5$)\label{prop:positivesuccess}
With the above definition, and the $\overline\epsilon$ chosen in Proposition \ref{prop:cactuslb},
there exists positive constants $c_{2}=c_{2}(\alpha,d)$ and $c_{3}=c(\alpha,d)$, $\overline{R}>0$
such that for all positive reals $r$ and $R$ that satisfy 
\begin{equation}
r>1\textrm{ and }R\geq\max(r^{d-2}/\overline{\epsilon},\overline{R})\label{eq:rRcondition2}
\end{equation}
one has 
\begin{equation}
\widetilde{\mathbb{P}}\left[H^{(Z)}(A^{s_{d}}(r,R))<T_{B^\circ(R^{2})}^{(Z)}\right]\geq c_{3}.\label{eq:synthetichittingproba}
\end{equation}
\end{prop}

\begin{proof}

We write $A$ for $A^{(s_{d})}(r,R)$ throughout this proof. 
On the one hand, by (\ref{eq:hitproba}),
\begin{equation}
\widetilde{\mathbb{P}}\left[H^{(Z)}_A<\infty\right]=\mathbb{E}_{x}^{(s_{d})}\left[\int g(z,y)e_{A}(dy)\right].
\end{equation}
By (\ref{inductionA}), 
for any $y\in A$,
\begin{equation}
g(z,y)\geq CR^{2-d}.
\end{equation}
By (\ref{eq:cactussdinf}) (note that condition (\ref{eq:rRcondition2})
is stronger than (\ref{eq:rRcondition})) we obtain that 
\begin{equation}
\widetilde{\mathbb{P}}\left[H^{(Z)}(A)<\infty\right]\geq CR{}^{2-d}\mathbb{E}_{x}^{(s_{d})}[{\rm cap}(A)]\stackrel{(\ref{eq:cactussdinf})}{\geq}c''R^{d-2}R^{2-d}=c_{4}.\label{eq:synproba1}
\end{equation}
On the other hand, by the strong Markov property of 
$(Z_t)_{t\geq 0}$, 
\begin{equation}
\begin{split}
&\quad\widetilde{\mathbb{P}}\left[\infty>H^{(Z)}_A\geq T^{(Z)}_{B^\circ(R^{2})}\right]
\leq\sup_{z'\in\partial B(R^{2})}P_{z'}^Z\otimes\mathbb{P}_{x}^{s_{d}}\left[H^{(Z)}_A<\infty\right]\\ \stackrel{(\ref{inductionA})}{\leq}&\sup_{z'\in\partial B(R^{2})}P_{z'}^Z[H_{B(cR)}<\infty],\label{eq:synproba2}
\end{split}
\end{equation}
moreover, 
\begin{equation}\begin{split}
&\sup_{z'\in\partial B(R^{2})}P^{Z}_{z'}[H_{B(cR)}<\infty]\\ \leq&\sup_{z'\in\partial B(R^{2}),y\in B(cR)}g(z',y){\rm cap}(B(cR))\stackrel{(\ref{eq:ballcap})}{\leq} cR^{4-2d}\cdot c'R^{d-2}\leq c_{5}R^{2-d}.
\end{split}\label{eq:synproba3}
\end{equation}
Hence the claim (\ref{eq:synthetichittingproba}) follows by combining
(\ref{eq:synproba1}),  (\ref{eq:synproba2}) and  (\ref{eq:synproba3}), choosing appropriate
$c_{3}$ and $\overline{R}$.
\end{proof}

As the last step before the final theorem, we set up a sequence of scales and show that almost surely, at some scale, $(Z_t)_{t\geq 0}$ hits the ``cactus-shaped'' set $A$ (in fact, we are able to show that such hitting happens at infinitely many scales).

For $x,z\in\mathbb{R}^{d}$, we define two sequences of positive
real numbers $(r_{k})_{k\geq0}$ and $(R_{k})_{k\geq0}$ through (see Propositions \ref{prop:cactuslb} and \ref{prop:positivesuccess} respectively for the definition of $\overline{\epsilon}$ and $\overline{R}$
\begin{equation}
r_{0}=\max(|x|,|z|)+1+2(s_{d}+1)\overline{R},\ R_{0}=\max(r_{0},\overline{\epsilon}^{-1}r_{0}^{d-2}),\label{eq:rR0}
\end{equation}
and for $k\geq1$ 
\begin{equation}
r_{k+1}=\max(R_{k}^{2},2(s_{d}+1)R_{k}),\ R_{k+1}=\overline{\epsilon}^{-1}r_{k+1}^{d-2}.\label{eq:rRk}
\end{equation}
Notice that for all $k\geq 0$ , conditions (\ref{eq:rRcondition2}) and  (\ref{eq:rRcondition}) are satisfied (when $r$ and $R$ are replaced by $r_k$ and $R_k$).
\begin{prop}\label{prop:definitivesuccess} ($d\geq 5$)For all $x,z\in\mathbb{R}^{d}$ and for $(r_{k})_{k\geq0}$ and $(R_{k})_{k\geq0}$
defined inductively in (\ref{eq:rR0}) and (\ref{eq:rRk}) (note that they depend implicitly on $x$ and $z$), one has
\begin{equation}
\widetilde{\mathbb{P}}\Big[\limsup_{k}\{H^{(Z)}\big(A^{s_{d}}(r_{k},R_{k})\big)<\infty\}\Big]=1.\label{eq:eventualhit}
\end{equation}
\end{prop}
\begin{proof}
We fix $x$ and $z$ throughout this proof. For the convenience of notation, we write $A_k$ for $A^{s_d}(r_k,R_k)$.
We first claim that, to prove (\ref{eq:eventualhit}), it suffices to prove for any
$g_{1},\ldots,g_{k-1}\in\{0,1\}$, 
\begin{equation}
\widetilde{\mathbb{P}}[\Gamma_{k}|\gamma_{1}=g_{1},\ldots,\gamma_{k-1}=g_{k-1}]\geq c>0,\label{eq:nomatterwhat}
\end{equation}
where
\begin{equation}\label{eq:Gammakdef}
\Gamma_{k}=\left\{H^{(Z)}(A^{s_{d}}(r_{k},R_{k}))\circ\theta^{(Z)}_{U}<T^{(Z)}(B^{\circ}(R_{k}^{2}))\circ\theta^{(Z)}_{U}\right\},
\end{equation}
(in which we write 
 and $U=T_{B^\circ(r_{k})}^{(Z)}$),
and $\gamma_{i}=1_{\Gamma_{i}}$, $i\in\mathbb{Z}^{+}$. 
In fact, by Borel-Cantelli lemma (see Lemma \ref{lem:Borel}), (\ref{eq:nomatterwhat})
implies that 
\begin{equation}
\widetilde{\mathbb{P}}\left[\limsup_{k}\Gamma_{k}\right]=1.\label{eq:nomatterwhat2}
\end{equation}
 Since the Brownian motion is transient when $d\geq3$, $T^{(Z)}(B^
\circ(R_{k}^{2}))<\infty$
, $\widetilde{\mathbb{P}}$-almost surely, (\ref{eq:nomatterwhat2}) implies that 
\begin{equation}
\limsup_{k}\left\{H^{(Z)}(A_k)<\infty\right\},\; \widetilde{\mathbb{P}}\mbox{-a.s}.
\end{equation}
This finishes the proof of (\ref{eq:eventualhit}) once we confirm (\ref{eq:nomatterwhat}).

Now we prove (\ref{eq:nomatterwhat}).
Pick $k\in\mathbb{Z}^{+}$,  and write 
\begin{equation}
\mathcal{F}_{k}=\sigma\left(Z_{t},\,0\leq t\leq T_{B^\circ(r_{k+1})}^{(Z)};\, X_{t},\,0\leq t\leq T_{B^\circ(r_{k+1})}^{(X)};\,\omega_{r,r_{k+1}}^{i},\, i=1,\ldots s_{d}\right).
\end{equation}
$({\cal F}_{k})_{k\geq1}$ forms a filtration. By (\ref{eq:rRk}) and (\ref{inductionA}), $\Gamma_k\in \mathcal{F}_k$.

It is straight-forward that for any $g_{1},\ldots,g_{k-1}\in\{0,1\}$, 
\begin{equation}
\left\{\gamma_{1}=g_{1},\ldots,\gamma_{k-1}=g_{k-1}\right\}\in{\cal F}_{k-1},
\end{equation}
Hence to prove (\ref{eq:nomatterwhat}) it suffices to prove that
\begin{equation}
\widetilde{\mathbb{P}}[\Gamma_{k}|\mathcal{F}_{k-1}]\geq c>0.\label{eq:nomatterwhat3}
\end{equation}
In fact, to benefit from Proposition \ref{prop:positivesuccess}, in the following calculation we are going to
integrate out $\omega_{r_{k}}^{i}$, since $\omega_{r_{k}}^{i}$, $i=1,\ldots,s_{d}-1$ are independent
from $\omega_{r_{k},\infty}^{i}$, $i=1,\ldots,s_{d}-1$, and then apply twice the strong Markov property, first to $(Z_{t})_{t\geq0}$ at time $U(=T_{B^\circ(r_{k})}^{(Z)})$ then to $(X_{t})_{t\geq 0}$
at time $V\stackrel{\triangle}{=}T^{(X)}_{B(r_{k})}$. We write $x'=X_V$ and $z'=Z_{V}$, and denote
by $P^{\omega_{r_{k},\infty}^{i}}$ the law of $\omega^{i}_{r_k,\infty}$ seen as the law of $\omega^{i}$ 
Now with the properties above, we know that
\begin{eqnarray*}
\widetilde{\mathbb{P}}[\Gamma_{k}|\mathcal{F}_{k-1}] & = & P^{Z}_{z}\otimes P_{x}^X\otimes\mathbb{P}^{(s_{d})}\left[\Gamma_{k}\big|{\cal F}_{k-1}\right]\\
 & \stackrel{(\ref{eq:nomatterwhat3})}{=} & P^{Z}_{z}\otimes P_{x}^X\bigotimes_{i=2}^{s_{d}}P^{\omega_{r_{k},\infty}^{i}}\left[\Gamma_{k}\big|\sigma(Z_{t},\,0\leq t\leq U;\, X_{t},\,0\leq t\leq V)\right]\\
 & \stackrel{{\rm Markov}}{=} & P^{Z}_{z'}\otimes P_{x}^X\bigotimes_{i=2}^{s_{d}}P^{\omega_{r_{k},\infty}^{i}} \left[H^{(Z)}_{A_k} < T^{(Z)}_{B^\circ(R_k^2)}\big|\sigma( X_{t},\,0\leq t\leq  V)\right]\\
 & \stackrel{{\rm Markov}}{\underset{(*)}{=}} & P_{z'}^Z \otimes P_{x'}^X\bigotimes_{i=2}^{s_{d}}P^{\omega_{r_{k},\infty}^{i}}\left[H^{(Z)}_{A_k} < T^{(Z)}_{B^\circ(R_k^2)}\right]\\
 & \stackrel{(\ref{eq:inde})}{=} & P^Z_{z'}\otimes P_{x'}^X\otimes\mathbb{P}^{(s_d)}[H^{(Z)}_{A_k} < T^{(Z)}_{B^\circ(R_k^2)}]\stackrel{(\ref{eq:synthetichittingproba})}{\geq} c>0,
\end{eqnarray*}
where in $(*)$ we use the fact that $A_k$ is independent from $\sigma(X_t,\,0\leq t\leq V)$.
This confirms (\ref{eq:nomatterwhat}) and finishes the proof of Proposition \ref{prop:definitivesuccess}.
\end{proof}
\begin{remark}\label{rem:graph distance}
From (\ref{eq:eventualhit}), we can conclude that that there exist $w^2,w^3,\ldots,w^{s_d}\in {\rm supp}(\omega^{(2)}+\cdots+\omega^{(s_d)})$, such that $X(\mathbb{R}_+)\cap B(w^2(\mathbb{R}),1)\neq \emptyset$, $w^2(\mathbb{R}) \cap B(w^3(\mathbb{R}),1)\neq \emptyset$,\ldots,  $w^{s_d}(\mathbb{R}) \cap B(Z(\mathbb{R}_+),1)\neq \emptyset$. Colloquially, this means one can connect $X$ and $Z$ via $s_d-1$ ``intermediate sausages''.
	\end{remark}
Once we have proved Proposition \ref{prop:definitivesuccess}, Proposition \ref{prop:gdub} follows from repeating the same argument as Lemma 4.13 and Theorem 1.1 from \cite{GDRI}, hence we omit the proof here.

Finally, Theorem \ref{graphdistance} follows by combining Propositions \ref{prop:graphlb} and \ref{prop:gdub} and applying scaling property of Brownian interlacements.

\section{Existence of Non-Trivial Phase Transition for the Vacant Set}\label{se:3}

This section is dedicated to Theorem \ref{thm:phase transition}, namely the existence of a non-trivial phase transition in percolation for the vacant set of Brownian interlacements.

We start with some comments on the strategy. The central object in this approach is the dyadic renormalization tree, introduced in Section \ref{se:3.1}, with a set of specific rules on how vertices of the tree should be embedded in $\mathbb{Z}^d$ so that the image of all leaves are in some sense ``well-separated''. 

The finiteness of $\alpha^*_1$ follows from the following energy-entropy competition. To be more specific, if 
the vacant set 
crosses an annulus with the size of order $6^n$, then 
 by a argument similar to \cite{Shortproof}, 
there exists $\mathcal{T}$, an embedding of the dyadic tree of depth $n$, 
such that the crossing path passes through the image of all $2^n$ leaves of $\mathcal{T}$, which are well-separated. On the one hand, we can show that the probability cost of the interlacement set avoiding all leaves is of order $\exp(-c\alpha 2^n)$, thanks to a capacity lower bound relying on the well-separation of leaves in the embedding. On the other hand, 
the number of possible embeddings is bounded by $C^{2^n}$ according to the rules. 
Hence, when $\alpha$ is sufficiently large, the crossing probability decays to $0$ very quickly as $n$ tends to infinity, implying that there is no chance of percolation.

To show the positiveness of $\alpha^*_1$ we focus on a plane in the Euclidean space and prove that when $\alpha$ is sufficiently small, $\mathcal{V}^\alpha_1$ percolates on this plane. Using planar duality, to this end, we only need to show that for a given large $L_0$, when $\alpha$ is small, the probability for the interlacement set to cross a planar annulus at scale $ L_0\cdot 6^n$ decays rapidly as $n$ tends to infinity. 
We are able to show that if such crossing takes place, the interlacement set will touch $2^n$ planar ``frames'' of size $L_0$, all of which are centered at the image of leaves of a planar embedding of the dyadic tree of depth $n$. The calculation of the probability of such event can be reduced to a large deviation estimate on the number of hitting of the frames by the trajectories from Brownian interlacements. Again, thanks to the well-separation of the leaves of dyadic renormalization tree,  if $L_0$ is chosen sufficiently large, such probability can be arbitrarily small by letting $\alpha$ tend to $0$, giving a bound strong enough to beat the combinatorial complexity. This shows that the crossing within the interlacement set is unlikely when $\alpha$ is small. Thus follows the positiveness of $\alpha^*_1$.

We now record notations we need later in this section. We denote
by $F$ the plane in $\mathbb{R}^d$ passing through the origin:
\begin{equation}
F=\mathbb{R}^{2}\times\{0\}^{d-2}\subset\mathbb{R}^{d},\label{eq:Fdef}
\end{equation}
and by $F_{Z}$ the discrete plane passing through the origin:
\begin{equation}
F_{Z}=\mathbb{Z}^{2}\times\{0\}^{d-2}=F\cap\mathbb{Z}^{d}.\label{eq:FZdef}
\end{equation}
We also denote the ``stick'' of length $L$ by
\begin{equation}
J_{L}=[0,L]\times\{0\}^{d-1}.
\end{equation}
For $x\in\mathbb{Z}^d$,  $R\in \mathbb{N}$, we denote by 
\begin{equation}
S(x,R)=\{y\in \mathbb{Z}^d :|y-x|_\infty= R\}
\end{equation}
the $l^\infty$-sphere (boundary of a discrete $l^{\infty}$-ball) centered at $x$ of radius $R$, and for $x\in F_Z$, by
\begin{equation}\label{eq:S2def}
S^{(2)}(x,R)=\{y\in F_Z :|y-x|_\infty= R\}
\end{equation}
the two dimensional discrete ``square'' centered at $x$ with ``size'' $R$.
From now on we fix
\begin{equation}
\beta=2\sqrt{d}+4.\label{eq:betachoice}
\end{equation}

The next lemma gives an upper bound on the capacity of the inflation of $S^{(2)}(x,L)$.
\begin{lem}
For all $L>3\beta$,
\begin{equation}
\mathrm{cap}(B(S^{(2)}(x,L),\beta))\leq\begin{cases}
cL & d\geq4\\
\frac{c'L}{\ln L} & d=3.
\end{cases}\label{eq:salamiupperbound}
\end{equation}
\end{lem}
\begin{proof}
Since $B(S^{(2)}(x,L),\beta)$ can be covered by four ``tubes'' which are $B(J_{L},\beta)$ after translation and rotation, 
to prove the claim (\ref{eq:salamiupperbound}) it suffices to show
that for all $L>3\beta$, 
\begin{equation}
\mathrm{cap}(B(J_{L},\beta))\leq\begin{cases}
cL & d\geq4\\
\frac{c'L}{\ln L} & d=3.
\end{cases}\label{eq:ubcap}
\end{equation}
Now we prove (\ref{eq:ubcap}). We consider a point 
$x=(x_1,x_2,\ldots,x_d)\in B(J_{L},\beta)$. Without loss of generality we assume that $x_{1}\geq L/2$. 
For all $y=(y_1,y_2,\ldots,y_d)\in B(J_{L},\beta)$ such that $y_1\leq L/3$, we know that
\begin{equation}
|x-y|\leq\sqrt{|x_{1}-y_{1}|^{2}+c\beta^{2}}\leq c'|x_{1}-y_{1}|.\label{eq:down to one dim}
\end{equation}
By (\ref{eq:capest}), to bound the capacity from above it suffices to establish a lower bound on an integral regarding the Green function. We write $D$ for ball of radius $\beta$ centered at $0$ in $(d-1)$ dimensions. We know that
\begin{eqnarray*}
\int_{y\in B(J_{L},\beta)}g(x,y)dy & \geq & \int_{[0,L/3]\times D}g(x,y)dy
  =  \int_{[0,L/3]\times D}c(|x-y|)^{2-d})dy\nonumber \\
 & \stackrel{(\ref{eq:down to one dim})}{\geq} & \int_{[0,L/3]}c'(|x_{1}-y_{1}|)^{2-d})dy_{1}\geq  \begin{cases}
c'' & d\geq4\\
\widetilde{c} \ln L & d=3.
\end{cases}
\end{eqnarray*}
By (\ref{eq:capest}), this readily implies (\ref{eq:ubcap}).
\end{proof}

\subsection{The Dyadic renormalization tree}\label{se:3.1}
In this section we construct the dyadic renormalization tree and state some of its useful properties. For readers' convenience we keep the same notation as in \cite{Shortproof}. 

For $n\geq0$, we write $T_{(n)=}\{1,2\}^{n}$ (in particular, $T_{(0)}=\emptyset$) for the index set of the $n$-th generation on the tree.
We denote by 
\begin{equation}
T_{n}=\bigcup_{k=0}^{n}T_{(k)}
\end{equation}
the dyadic tree of depth $n$. For $0\leq k\leq n$, and a node $m=(\xi_{1},\ldots,\xi_{k})\in T_{(k)}$,
we write 
\begin{equation}
m_{1}=(\xi_{1},\ldots,\xi_{k},1)\textrm{ and }m_{2}=(\xi_{1},\ldots,\xi_{k},2),
\end{equation}
for the two children of $m$ which lie in $T_{(k+1)}.$ Given an integer $L_{0}\geq1$ (we
will specify our choice of $L_{0}$ at the beginning of each subsection) we write down a sequence of scales
\begin{equation}\label{eq:Lnchoice}
L_{n}=L_{0}\cdot6^{n},\ n\geq0.
\end{equation}
For $n\geq0$, we also denote by ${\cal L}_{n}=L_{n}\mathbb{Z}^{d}$
the integer lattice renormalized by $L_{n}$. 

We call ${\cal T}:T_{n}\to\mathbb{Z}^{d}$ (resp.~$F_{Z}$) a proper
embedding of $T_{n}$ into $\mathbb{Z}^{d}$ (resp.~$F_{Z}$) rooted
at $x\in{\cal L}_{n}$ (resp. ${\cal L}_{n}\cap F_{Z}$), if one has
\begin{equation}
{\cal T}(\emptyset)=x,
\end{equation}
plus for all $0\leq k\leq n$, and $m\in T_{(k)}$, 
\begin{equation}
{\cal T}(m)\in{\cal L}_{n-k}\textrm{ (resp. }{\cal L}_{n-k}\cap F_{Z}),
\end{equation}
and moreover for all $0\leq k<n$, and $m\in T_{(k)}$, 
\begin{equation}\label{distanceprop}
|{\cal T}(m_{1})-{\cal T}(m)|_{\infty}=L_{n-k}\textrm{ and }|{\cal T}(m_{2})-{\cal T}(m)|_{\infty}=2L_{n-k}.
\end{equation}
For $x\in{\cal L}_{n}$, one writes $\Lambda_{n,x}$ (resp. $\Lambda_{n,x}^{F}$)
the set of proper embeddings of $T_{n}$ into $\mathbb{Z}^{d}$ (resp.
$F_{Z}$) with root $x$.

Now, for the sake of completeness, we quote three lemmas without proof (Lemmas 3.2-3.4 in \cite{Shortproof}) on dyadic trees and their embeddings.

\vspace{0.3cm}
The following lemma counts the total number of embeddings into of
$T_{n}$ into $\mathbb{Z}^{d}$ and $F_Z$.
\begin{lem}\label{set1}
For $d\geq 2$,  $L_{0}\geq1$, $n\geq0$ and $x\in{\cal L}_{n},$ there exists $C=C(d)>0$ (which does not depend on $x$) such that 
\begin{equation}\label{eq:set1}
|\Lambda_{n,x}|=C^{2^{n}-1}.
\end{equation}
In particular, for any $L_0\geq 1$, $n\geq 0$ and $x \in {\cal L}_n \cap F_Z$, $|\Lambda^{F}_{n,x}|=C(2)^{2^{n}-1}$.
\end{lem}
The following lemma shows that if for all $n\geq0$ a $*$-path (see the beginning of Section \ref{notations} for the precise definition) 
goes
through $S(x,L_{n})$ and $S(x,2L_{n})$, then there is a proper embedding
of $T_{n}$ on $\mathbb{Z}^{d}$ (for all $d\geq 2$) such that
every leaf of this tree sits ``on'' the path. Note that this lemma applies to $F_Z$ as well for it is equivalent to the case of $d=2$.
\begin{lem}\label{set2}
For any $L_{0}\geq1$, $n\geq0$ and $x\in{\cal L}_{n}$, if $\gamma$ is a $*$-path (and in particular a nearest neighbor path) in $\mathbb{Z}^{d}$, $d\geq2$, 
such that 
\begin{equation}
\gamma\cap S(x,L_{n}-1)\neq\emptyset\textrm{ and }\gamma\cap S(x,2L_{n})\neq\emptyset,
\end{equation}
then there exists ${\cal T}\in\Lambda_{n,x}$ such that 
\begin{equation}\label{eq:set2conclu}
\gamma\cap S({\cal T}(m),L_{0}-1)\neq\emptyset\,\textrm{ for all }m\in T_{(n)}.
\end{equation}
\end{lem}
To state the next lemma we need extra notations. For $0\leq k\leq n$
and $m=(\xi_{1},\ldots,\xi_{n})\in T_{(n)}$, we denote by $m|_{k}$
the projection $(\xi_{1},\ldots,\xi_{k})\in T_{(k)}$. For $m,m'\in T_{(n)}$,
we define the lexical distance between $m$ and $m'$ through
\begin{equation}
\rho(m,m')=\min\{k\geq0:m|_{n-k}=m'|_{n-k}\}.
\end{equation}
 For any $m\in T_{(n)}$ we denote by 
\begin{equation}
T_{(n)}^{m,k}=\{m'\in T_{(n)}:\rho(m,m')=k\}
\end{equation}
all the leaves with lexical distance $k$ from $m$. Note that for all
$1\leq k\leq n$, 
\begin{equation}\label{numeachlevel}
|T_{(n)}^{m,k}|=2^{k-1}.
\end{equation}
 The following lemma shows that a proper embedding is relatively ``spread-out'' on all scales and will be used in the proofs of Propositions \ref{prop:ubbi} and \ref{prop:lb}.
\begin{lem}\label{spreadout}
For all $n\ge1$, $x\in{\cal L}_{n}$, ${\cal T}\in\Lambda_{n,x},$
$m\in T_{(n)}$, $k\geq1$, and for all $m'\in T_{(n)}^{m,k}$, $y\in S({\cal T}(m),L_{0}-1),$
$z\in S({\cal T}(m'),L_{0}-1)$, one has
\begin{equation}\label{eq:spreadout}
|y-z|\geq L_{k-1}.
\end{equation}

\end{lem}
\subsection{Preliminary results for the upper bound on the threshold}\label{se:3.2}
In this subsection we prepare all the ingredients for the proof of the first part of (\ref{eq:pt}), namely the finiteness of the percolation threshold $\alpha^*_1$.
We fix, only in this subsection,
\begin{equation}
L_{0}=1.\label{eq:L0choiceub}
\end{equation}
We now assign symbols to the crossing events that we consider in this subsection. For all $\alpha>0$, and $n\in\mathbb{N}$, we write (see the beginning of Section \ref{notations} for the definition of a continuous path and (\ref{eq:betachoice}) for the definition of $\beta$)
\begin{equation}\label{Analphadef}
A_{n}^{\alpha}=\{\exists\textrm{ a continuous path in }{\cal V}_{\beta}^{\alpha}\textrm{ connecting }B_\infty(0,L_{n}-1)\textrm{ and }\partial B_\infty(0,2L_{n})\},
\end{equation}
and 
\begin{equation}\label{eq:Dandef}
D_{n}^{\alpha}=\{\exists \textrm{ nearest-neighbour path in }{\cal V}_{1}^{\alpha}\cap\mathbb{Z}^d\textrm{ connecting }S(0,L_{n}-1)\textrm{ and }S(0,2L_{n})\}.
\end{equation}
The following lemma shows that actually $A_{n}^{\alpha}$ is almost surely contained
in $D_{n}^{\alpha}$. As it is almost obvious, we omit its proof.
\begin{lem}
\label{lem:reducetostar}For all $\alpha>0$, and $n\in\mathbb{N}$, 
\begin{equation}\label{inclu}
A_{n}^{\alpha}\subseteq D_{n}^{\alpha},
\end{equation}
and hence
\begin{equation}
\mathbb{P}[A_{n}^{\alpha}]\leq\mathbb{P}[D_{n}^{\alpha}].\label{eq:reducetostar}
\end{equation}
\end{lem}

We now state and prove the main result in this subsection. 
\begin{prop}
\label{prop:ubbi}There exists $\alpha^{\#}>0$, such that for all
$\alpha>\alpha^{\#}$, and for any $n\geq1$, 
\begin{equation}
\mathbb{P}[A_{n}^{\alpha}]\leq\left(\frac{1}{2}\right)^{2^{n}}.\label{eq:A_nub}
\end{equation}
\end{prop}

\begin{proof}
Recall that $L_0=1$.
By Lemma \ref{spreadout}, on the event $D_{n}^{\alpha}$, there is at least one proper embedding of $T_{n}$, which we denote by $\widehat{{\cal T}}\in \Lambda_{n,0}$, such that there is a nearest neighbor path $\widetilde{\gamma}:[0,\ldots,M]\to\mathbb{Z}^d$, for some $M\in\mathbb{N}$, which lies in $\mathcal{V}^\alpha_1$, i.e., $\widetilde{\gamma}([0,\ldots,M])\in\mathcal{V}^\alpha_1\cap\mathbb{Z}^d$, that connects $S(0,L_{n}-1)$ to $S(0,2L_n)$ (see the definition
of $D_{n}^{\alpha}$ in (\ref{eq:Dandef})) and passes through all the points in $\widehat{{\cal T}}$.
This implies that $\cup_{m\in T_{(n)}}\widehat{{\cal T}}(m)$ is not touched
by ${\cal I}_{1}^{\alpha}$. 
By (\ref{simplechar}), one has that
\begin{equation} 
\mathbb{P}[\cup_{m\in T_{(n)}}\widehat{{\cal T}}(m)\cap{\cal I}_{1}^{\alpha}=\emptyset]=\exp(-\alpha\mathrm{cap}({\cal X}_{
\widehat{\cal T}}))
\end{equation}
where for $\mathcal{T}\in\Lambda_{n,0}$ we define
\begin{equation}
{\cal X}_{{\cal T}}=\cup_{m\in T_{(n)}}B({\cal T}(m),1).
\end{equation}
By Lemma \ref{set1}, there are at most $C^{2^{n}}$ possible embeddings of $T_{n}$, 
hence we know that
\begin{equation}
\mathbb{P}[A_{n}^{\alpha}]\stackrel{(\ref{eq:reducetostar})}{\leq}\mathbb{P}[D_{n}^{\alpha}]\leq C^{2^{n}}\max_{{\cal T}\in\Lambda_{n,0}}\exp(-\alpha\mathrm{cap}({\cal X}_{{\cal T}})).\label{eq:A_nprelimbound}
\end{equation}
Now, we claim that (\ref{eq:A_nub}) is proved if we can prove that for all ${\cal T}\in\Lambda_{n,0}$,
\begin{equation}
{\rm cap}({\cal X}_{{\cal T}})\geq c2^{n},\label{eq:capXT1lb}
\end{equation}
uniformly for some $c>0$. This is because with
(\ref{eq:capXT1lb}), if one chooses a sufficiently large $\alpha^{\#}>0$, then for all $\alpha >\alpha ^\#$, 
\begin{equation}
\mathbb{P}[A_{n}^{\alpha}]\stackrel{(\ref{eq:A_nprelimbound})}{\leq}(C\exp(-c\alpha))^{2^{n}}\leq(1/2)^{2^{n}}.
\end{equation}

Now we prove (\ref{eq:capXT1lb}). Thanks to  (\ref{eq:capest}) and the fact that $|{\cal X_T}|=c2^n$,
to give a lower bound on the capacity of $\cal{X_T}$ it suffices to bound from above the denominator of the fraction in the first term of (\ref{eq:capest}), i.e.~an integral of the Green function. In fact, for all $m\in T_{(n)}$ and $x\in B(\mathcal{T}(m),1)$ (note that such $x$ runs over $\cal{X_T}$), we have
%
\begin{eqnarray}
 \int_{y\in {\cal X_T}} g(x,y)dy& =& \sum_{m'\in T_{(n)}}\int_{B(0,1)}g({\cal T}(m)+x,{\cal T}(m')+y)dy\nonumber \\
 & = & \sum_{k=0}^{n}\sum_{m'\in T_{(n)}^{m,k}}\int_{B(0,1)}g({\cal T}(m)+x,{\cal T}(m')+y)dy\\
 &  \stackrel{(\ref{eq:spreadout})}{\leq}& c+c'\sum_{k=1}^{n}L_{k-1}^{2-d}\big|T_{(n)}^{m,k}\big| \stackrel{(\ref{numeachlevel}),(\ref{eq:Lnchoice})}{\leq}c''.\nonumber 
\end{eqnarray}
This finishes the proof of (\ref{eq:capXT1lb}) as well as (\ref{eq:A_nub}).
\end{proof}

\subsection{Preliminary results for the lower bound on the threshold}
\label{se:3.3}
In this subsection we prove some preparatory results for proof of the second
half of (\ref{eq:pt}), namely the positiveness of $\alpha_{1}^{*}$. We postpone the choice of $L_0$ until the end of the proof of Proposition \ref{prop:lb}.

We now assign symbols to the crossing events we consider in this subsection. For $\alpha>0$
and $n$ in $\mathbb{N}$, we note the following event (recall the definition of $L_n$ in (\ref{eq:Lnchoice}))
\begin{equation}
\widehat{A}_n^{\alpha}=\{\textrm{there is a continuous path in }{\cal V}_{1}^{\alpha}\textrm{ connecting }B_{\infty}(0,L_{n})\textrm{ and infinity}\}.
\end{equation}
Here  by ``a continuous path in ${\cal V}_{1}^{\alpha}$ connecting $B_{\infty}(0,L_{n})$ and infinity'' we mean a continuous $\gamma:[0,\infty)\to \mathbb{R}^d$ such that $\gamma[0,\infty)\subset {\cal V}_{1}^{\alpha}$ and $\limsup_{t\to\infty}|\gamma(t)|=\infty$.

We also define the following events for $x\in\mathbb{R}^d$, $\alpha>0$ and $k$ in
$\mathbb{N}$, (recall the definition of $F$ in (\ref{eq:Fdef})), 
\begin{equation}
B_{k,x}^{\alpha}=\{\textrm{there is a continuous path in }{\cal I}_{1}^{\alpha}\cap F\textrm{ connecting }B_{\infty}(x,L_{k})\textrm{ and }\partial B_{\infty}(x,2L_{k})\}
\end{equation}
and for all $x\in F_Z$ (recall the definition of $S^{(2)}$ in (\ref{eq:S2def})), 
\begin{equation}
\overline{B}_{k,x}^{\alpha}=\{\exists\textrm{ nearest-neighbour path in }{\cal I}_{\beta}^{\alpha}\cap F_{Z}\textrm{ connecting }S^{(2)}(x,L_{k}-1)\textrm{ and }S^{(2)}(x,2L_{k})\}.
\end{equation}

The next lemma shows that almost surely $B_{k,x}^{\alpha}$ is contained in $\overline{B}_{k,x}^{\alpha}$.  We omit its proof due to similarity with the proof of Lemma \ref{lem:reducetostar}. 
\begin{lem}\label{Binclusion}
For all $\alpha>0$,
$k\in\mathbb{N}$, $x\in F_Z$ one has that 
\begin{equation}\label{eq:Bincl}
B_{k,x}^{\alpha}\subset\overline{B}_{k,x}^{\alpha}.
\end{equation}
\end{lem}

The following lemma relates $\widehat{A}_n^{\alpha}$ with $B_{k,x}^{\alpha}$ by asserting that on the event $\big(\widehat{A}_{n}^{\alpha}\big)^{c}$, $B_{k,x}^{\alpha}$ must happen for a fixed number of choices of $x$.
\begin{lem}\label{fakedec}
For all $n\geq 2$,  
\begin{equation}\label{eq:fakedec}
 \big(\widehat{A}_{n}^{\alpha}\big)^{c} \subset \bigcup_{k\geq n-1} \bigcup_{x\in {\cal L}_{k-1}\cap F_Z,\;|x|_{\infty}\leq4L_{k+1}} B_{k,x}^{\alpha}.
\end{equation}
\end{lem}
\begin{proof}
On the event $\big(\widehat{A}_{n}^{\alpha}\big)^{c}$, $S(0,L_{n})$ is completely surrounded by a continuous path $\gamma:[0,1]\to F$, such that 
\begin{equation}
\gamma(0)=\gamma(1)\quad\mbox{ and }\quad\gamma([0,1])\subset{\cal I}_{1}^{\alpha}\cap F.
\end{equation}
Now suppose $\gamma\subseteq B_\infty(0,4L_{k+1})$ but $\gamma \nsubseteq B_\infty(0,4L_{k})$ for
some $k\geq n-1$. We then consider 
\begin{equation}
Q_{k}=\{x\in\mathcal{L}_{k-1}:\gamma\cap S(x,L_{k})\neq\emptyset\}.
\end{equation}
Pick $x_{0}\in Q_{k}$. We claim that
\begin{equation}\label{eq:fakedeccontra}
\gamma\nsubseteq B_\infty(x_{0},2L_{k})
\end{equation}
hence there is a continuous path from $S(x_0,L_{k})$ to $S(x_0,2L_{k})$.
Now, suppose (\ref{eq:fakedeccontra}) is not true. Then the $l^\infty$-diameter of $\gamma$ must be smaller or equal to $4L_{k}$. But if this is the case, since on $F$ the origin is surrounded by $\gamma$,  we would have $\gamma\subseteq B_\infty(0,4L_{k})$, a contradiction! This finishes the proof of (\ref{eq:fakedec}).
\end{proof}

This subsection culminates in the following proposition, which almost immediately implies the second half of Theorem \ref{thm:phase transition}, as we will see in Section \ref{denouement}.

\begin{prop}\label{prop:lb}
There exist $L_0,\widehat{\alpha}>0$, such that for all $\alpha\in(0,\widehat{\alpha})$ and all $x\in F_Z$,
\begin{equation}
\mathbb{P}\left[\overline{B}_{n,x}^{\alpha}\right]\leq(1/4)^{2^{n}}\label{eq:Bbarub}
\end{equation}
holds for large $n$.
\end{prop}
\begin{proof}
We postpone the choice of $L_{0}$ till the end of this proof. 

For our convenience we write $\square_x$ for $S^{(2)}(x,L_0-1)$. Without loss of generality we take $x=0$. For ${\cal T}_F\in\Lambda_{n,0}^{F}$ (see below (\ref{distanceprop}) for the definition of $\Lambda_{n,0}^{F}$), we write 
\begin{equation}
K(\mathcal{T}_F)=\bigcup_{m\in T_{(n)}}\square_{\mathcal{T}(m)},
\end{equation}
which is a subset of $F$, and
\begin{equation}\label{eq:Kprimedef}
K'(\mathcal{T}_F)=B(K(\mathcal{T}_F),\beta),
\end{equation}
which is a subset of $\mathbb{R}^d$.
By Lemmas \ref{set1} and \ref{set2}, one has that
\begin{equation}\label{eq:Bba2}
\begin{split}
\mathbb{P}[\overline{B}_{n,0}^{\alpha}]&\stackrel{(\ref{eq:set2conclu})}{\leq}
\mathbb{P}\left[\bigcup_{{\cal T}_F\in\Lambda_{n,0}^{F}}\bigcap_{m\in T_{(n)}}\big\{\square_{\mathcal{T}_F(m)} \cap{\cal I}_{\beta}^{\alpha}\neq\emptyset\big\}\right]\\ 
&\stackrel{(\ref{eq:set1})}{\leq} C^{2^{n}}\max_{{\cal T}_F\in\Lambda_{n,0}^{F}}\mathbb{P}\left[\bigcap_{m\in T_{(n)}}\big\{\square_{\mathcal{T}_F(m)} \cap{\cal I}_{\beta}^{\alpha}\neq\emptyset\big\}\right].
\end{split}
\end{equation}
Hence, to prove (\ref{eq:Bbarub}), it suffices to derive an adequate upper bound on 
\begin{equation}
\max_{{\cal T}_F\in\Lambda_{n,0}^{F}}\mathbb{P}\left[\bigcap_{m\in T_{(n)}}\big\{\square_{\mathcal{T}_F(m)} \cap{\cal I}_{\beta}^{\alpha}\neq\emptyset\big\}\right].
\end{equation}
Let ${\cal T}_F\in\Lambda_{n,0}^{F}$ be
an embedding of $T_{n}$ into $F_{Z}$. We drop the dependence on ${\cal T}_F$ in notation whenever there is no confusion arising.
For $w\in W^+$,
we write by ${\cal N}(w)$ the number of frames (i.e.~$\square_{\cdot}$) with
centers in ${\cal T}_F(T_{(n)})$ (``leaves'' of the embedding of dyadic tree) 
which are hit by the sausage of $w$ with radius $\beta$, i.e, 
\begin{equation}
{\cal N}(w)=\Big|\{m\in T_{(n)},\;B(w([0,\infty)),\beta)\cap\square_{\mathcal{T}_F(m)} \neq\emptyset\}\Big|.
\end{equation}
Now, if ${\cal I}_{\beta}^{\alpha}$ intersects with all the frames  on the leaf level $T_{(n)}$, then the total count of hits must be at least $2^n$, i.e.,
\begin{equation}\label{eq:Bba3}
\mathbb{P}\left[\bigcap_{m\in T_{(n)}}\big\{\square_{\mathcal{T}_F(m)} \cap{\cal I}_{\beta}^{\alpha}\neq\emptyset\big\}\right]\leq \mathbb{P}\left[\sum_{j=1}^{N}{\cal N}(w_j)\geq2^{n}\right],
\end{equation}
where $N$ is determined through
\begin{equation}
\mu_{K',\alpha}(\omega)=\sum_{j=1}^{N}\delta_{(w_j,\alpha_j)}
\end{equation}
(see (\ref{eq:mukadef}) and (\ref{eq:Kprimedef}) for notation). Now we write 
\begin{equation}
p= \max_{m\in T_{(n)},\;y\in B(\square_{{\cal T}_F(m)},\beta)} P_{y}\Big[X_{[0,\infty)}\cap B(K\backslash \square_{{\cal T}_F(m))}, \beta) \neq \emptyset\Big].
\end{equation}
By the strong Markov property of Brownian motion, one knows that 
\begin{equation}\label{eq:puse}
P_{\widetilde{e}_K}[{\cal N}(X)\geq k]\leq p^{k-1}
\end{equation}
for any $k\geq 1$.
We are also able show that $p$ can be taken arbitrarily small if we take $L_0$ sufficiently large. In fact,  for any $m\in T_{(n)}$ and $y\in B(\square_{{\cal T}_F(m)},\beta)$,
\begin{eqnarray}
p&\leq & \max_{m\in T_{(n)},\;y\in B(\square_{{\cal T}_F(m)},\beta)}\sum_{k=1}^{n}\sum_{m'\in T_{(n)}^{m,k}}P_{y}[X_{[0,\infty)}\cap B(\square_{{\cal T}_F(m')},\beta)\neq\emptyset]\nonumber \\
 & \stackrel{(\ref{eq:spreadout}),(\ref{eq:greenvalue})}{\underset{(\ref{eq:hitproba})}{\leq}} & \sum_{k=1}^{n}\sum_{m'\in T_{(n)}^{m,k}}cL_{k-1}^{2-d}{\rm cap}\Big(B(\square_{{\cal T}_F(m')},\beta)\Big)\label{eq:puse2}\\
 & \stackrel{(\ref{numeachlevel})}{\leq} &  \sum_{k=1}^{n} (6^{2-d}\cdot 2)^{k-1}L_{0}^{2-d}c{\rm cap}\Big(B(\square_{0},\beta)\Big)\stackrel{(\ref{eq:salamiupperbound})}{\leq} \begin{cases}
c'L_0^{3-d} & d\geq4\\
c''(\ln L_0)^{-1} & d=3
\end{cases}\stackrel{\triangle}{=}q_d(L_0).\nonumber 
\end{eqnarray}

With the help from (\ref{eq:puse}) and (\ref{eq:puse2}), by the same argument (involving exponential Chebyshev inequality) as in the proof of Proposition 5.1 in \cite{Shortproof} (see below (5.7) in \cite{Shortproof}), 
we know that 
\begin{equation}\label{eq:puse3}
\mathbb{P}[\overline{B}_{n,x}^{\alpha}]\stackrel{(\ref{eq:Bba2})}{\underset{(\ref{eq:Bba3})}{\leq}}C^{2^n}\max_{{\cal T}\in\Lambda_{n,0}^{F}}\mathbb{P}\Big[\sum_{j=1}^{N}{\cal N}(w_j)\geq2^{n}\Big]\leq 
\Big(Cq_d(L_0)\Big)^{2^n}\left[\exp\bigg(\alpha \frac{\mathrm{cap}(B(\square_0,\beta)}{q_d(L_0)}\bigg)\right]^{2^n}. 
\end{equation}

We first fix $L_0$  sufficient large, such that $Cq_d(L_0)<1/8$, and then choose $\widehat{\alpha}$ sufficiently small such that for all $\alpha \in (0,\widehat{\alpha})$, $\exp\left(\alpha \frac{\mathrm{cap}(B(\square_0,\beta)}{q_d(L_0)}\right)<2$. With this choice of $L_0$ and $\widehat{\alpha}$, we know that the right hand side of (\ref{eq:puse3}) is bounded above from $(1/4)^{2^n}$. This finishes the proof of (\ref{eq:Bbarub}).
\end{proof}
\subsection{Denouement}\label{denouement}
In this subsection we use Propositions \ref{prop:ubbi} and \ref{prop:lb} to prove Theorem \ref{thm:phase transition}.
\begin{proof}[Proof of Theorem \ref{thm:phase transition}]
We start by showing that when $\alpha$ is sufficiently large, the vacant set $\mathcal{V}^\alpha_\beta$ does not percolate. 
By the definition of $A_n^\alpha$ (see (\ref{Analphadef})) one knows that for $\alpha>\alpha^{\#}$ and for all $M\geq 0$,
\begin{equation}\label{denu1}
\mathbb{P}\Big[{\cal V}_{\beta}^{\alpha}\mbox{ percolates}\Big]\leq\mathbb{P}\left[\bigcup_{n\geq M} A_{n}^{\alpha}\right]
\end{equation}
and by Proposition \ref{prop:ubbi},
\begin{equation}
\mathbb{P}\Big[\cup_{n\geq M} A_{n}^{\alpha}\Big]\leq \sum_{n\geq M} (1/2)^{2^{n}}\to0\mbox{ as $M$ tends to infinity}.
\end{equation}
This means that 
\begin{equation}
\inf\Big\{\alpha\geq 0:\,\mbox{$\mathcal{V}^\alpha_\beta$ does not percolate a.s.}\Big\}\leq \alpha^\#.
\end{equation}
By the scaling property (see (\ref{scaling})) of Brownian interlacements, we obtain that
\begin{equation}
\alpha^*_1\stackrel{\triangle}{=}\inf\Big\{\alpha\geq 0:\,\mbox{$\mathcal{V}^\alpha_1$ does not percolate a.s.}\Big\}\leq\alpha^{\#}\beta^{d-2}<\infty,
\end{equation}
and that for all $r>0$,
 \begin{equation}
  \alpha^*_r=\alpha^*_1r^{2-d}=\inf\{\alpha\geq 0:\,\mbox{$\mathcal{V}^\alpha_r$ does not percolate a.s.}\},
  \end{equation}
proving the first part of (\ref{eq:pt}).

We then claim that
\begin{equation}\label{eq:denu2}
\mbox{$\mathcal{V}^\alpha_1$ percolates almost surely, when  $\alpha<\widehat{\alpha}$}
\end{equation}  
(see the statement of Proposition \ref{prop:lb} for the definition of $\widehat{\alpha}$), which readily implies 
the second part of (\ref{eq:pt}).

We now prove (\ref{eq:denu2}).  For any $M\geq 0$
\begin{equation}\label{eq:notpercolate}
\mathbb{P}[{\cal V}_{\beta}^{\alpha}\mbox{ does not percolate}]\leq\mathbb{P}\left[\bigcap_{n\geq M} (\widehat{A}_{n}^{\alpha})^c\right].
\end{equation}
Since $(\widehat{A}_{n}^\alpha)_{n\geq 0}$ is a sequence of increasing events, to prove (\ref{eq:denu2}) it suffices to show that
for all $\alpha<\widehat{\alpha}$,
\begin{equation}\label{eq:sh1}
\lim_{n\to\infty}\mathbb{P}\left[\big(\widehat{A}_{n}^{\alpha}\big)^{c}\right]=0.
\end{equation}
Now we prove (\ref{eq:sh1}). One first notices that by Lemma \ref{fakedec},
\begin{equation}\label{eq:sh2}
\mathbb{P}\left[\big(\widehat{A}_{n}^{\alpha}\big)^{c}\right]\leq\mathbb{P}\left[\bigcup_{k=n-1}^{\infty}\bigcup_{x\in{\cal L}_{k-1}\cap F_Z,\,|x|_\infty\leq4L_{k+1},}B_{k,x}^{\alpha}\right].
\end{equation}

By Lemma \ref{Binclusion}, 
 the fact that $|{\cal L}_{k-1}\cap S^{(2)}(0,4L_{k+1})|\leq C=C(d)$ and Proposition \ref{prop:lb}, one obtains that for $\alpha<\widehat{\alpha}$,
\begin{equation}
\mathbb{P}\left[\big(\widehat{A}_{n}^{\alpha}\big)^{c}\right]\stackrel{(\ref{eq:sh2})}{\underset{(\ref{eq:Bincl})}{\leq}}\mathbb{P}\left[\bigcup_{k=n-1}^{\infty}\bigcup_{x\in{\cal L}_{k-1}\cap F_Z,\,|x|_\infty\leq 4L_{k+1},}\overline{B}_{k,x}^{\alpha}\right]
\stackrel{(\ref{eq:Bbarub})}{\leq}\sum_{k=n}^{\infty}C(1/4)^{2^{k-1}}\to0\textrm{ as }n\to\infty.
\end{equation}
 This completes the proof of (\ref{eq:sh1}) and 
finishes the proof of (\ref{eq:pt}).

\end{proof}
\begin{remark}\label{rem:endremark}
1) Recall the definition of $\alpha^{**}_r$, $\alpha^{\#}$ and $\beta$ respectively in (\ref{eq:astarstar}), Proposition \ref{prop:ubbi} and (\ref{eq:betachoice}). It follows from Proposition \ref{prop:ubbi}, the definition of events $A_n^\alpha$ (see (\ref{Analphadef})), and scaling property of Brownian interlacements that
\begin{equation}
\big(0<\big)\alpha_r^*\leq\alpha_r^{**}\leq \alpha^{\#}(\beta/r)^{d-2}\big(<\infty\big).
\end{equation}It is a natural question whether the two thresholds $\alpha_r^*$ and $\alpha_r^{**}$ coincide, or, in other words, whether the phase transition for the vacant set is sharp. As the corresponding conjecture in the case of random interlacements still remains open, we speculate that in our case, this question is also not easy to answer. It might be even harder to answer what happens at these critical values, e.g., whether $\mathcal{V}^{\alpha^{*}_1}_{1}$ percolates or not. 

\noindent  2) Note that as a byproduct of the proof of Theorem \ref{thm:phase transition}, we obtain that when $\alpha<\widehat{\alpha}$, $\mathcal{V}^\alpha_1$ percolates not only in the whole space, but in a plane and in a slab  $F_R=\mathbb{R}^2 \times [0,R]^{d-2}$  (where $R>0$ stands for the ``thickness'' ) as well. With this observation in mind we define for $r>0$
\begin{equation}
\widetilde{\alpha}_r=\sup\Big\{\alpha\geq 0: \mathbb{P}[\mathcal{V}^{\alpha}_r\mbox{ percolates in a plane}]=1\Big\}
\end{equation} 
and
\begin{equation}
\widetilde{\alpha}_{R,r}=\sup\Big\{\alpha\geq 0: \mathbb{P}[\mathcal{V}^{\alpha}_r\mbox{ percolates in $F_R$}]=1\Big\}.
\end{equation} 
It follows that
\begin{equation}
\big(\widehat{\alpha} r^{2-d}\leq\big)\widetilde{\alpha}_r\leq \widetilde{\alpha}_{R,r}\leq \alpha^*_r,
\end{equation} 
By an intuitive analogy to the case of Bernoulli percolation (see e.g.~Sections 7.1 and 7.2 in \cite{Grimmett}), we think it's very plausible that $\widehat\alpha r^{2-d}=\widetilde\alpha_r$, while $\widetilde{\alpha}_r<\alpha^*_r$ but $\lim_{R\to\infty}\widetilde{\alpha}_{R,r}= \alpha^*_r$. Techniques developed in \cite{Sharpperco} may be helpful for answering the first conjecture.

\noindent 3) We are also prompted to ask whether the unbounded cluster in $\mathcal{V}^\alpha_\beta$ is unique 
in the supercritical regime and wonder if it is possible to adapt the proof of the uniqueness of infinite cluster in the vacant set of random interlacements in \cite{Uniqueness} to tackle this problem.
\end{remark}
\end{document}